\documentclass[11pt]{article}
\usepackage{amssymb,amsmath,epsfig}
\usepackage[colorlinks=false,breaklinks=true,linkcolor=blue]{hyperref}
\usepackage{graphics,psfrag,graphicx,color}
\usepackage{diagbox}
\usepackage{float,hhline}

\numberwithin{equation}{section}
\newcommand{\ds}{\displaystyle}

\newcommand{\bsi}{{\boldsymbol\sigma}}

\newcommand{\btau}{{\boldsymbol\tau}}

\newcommand{\bv}{{\mathbf{v}}}
\newcommand{\bw}{{\mathbf{w}}}
\newcommand{\f}{\mathbf{f}}
\newcommand{\g}{\mathbf{g}}
\newcommand{\ba}{\mathbf{a}}
\newcommand{\bb}{\mathbf{b}}

\newcommand{\bi}{\mathbf{i}}

\newcommand{\bu}{\mathbf{u}}

\newcommand{\bz}{{\mathbf{z}}}
\newcommand{\bt}{{\mathbf{t}}}
\newcommand{\bn}{{\mathbf{n}}}

\newcommand{\0}{{\mathbf{0}}}

\def\bK{\mathbf{K}}
\def\bI{\mathbb{I}}

\def\bV{\mathbf{V}}

\def\bQ{\mathbf{Q}}

\def\bRT{\mathbf{RT}}
\def\bBR{\mathbf{BR}}

\def\bx{\mathbf{x}}

\def\bz{\mathbf{z}}
\newcommand{\bL}{\mathbf{L}}
\newcommand\bH{\mathbf{H}}

\newcommand\bbI{\mathbb{I}}

\newcommand\bbH{\mathbb{H}}

\newcommand\bbL{\mathbb{L}}

\newcommand{\cA}{\mathcal{A}}
\newcommand{\cB}{\mathcal{B}}
\newcommand{\cC}{\mathcal{C}}

\newcommand{\cT}{\mathcal{T}}

\newcommand{\cM}{\mathcal{M}}
\newcommand{\cD}{\mathcal{D}}
\newcommand{\cO}{\mathcal{O}}
\newcommand{\cP}{\mathcal{P}}

\def\R{\mathrm{R}}

\def\H{\mathrm{H}}

\def\L{\mathrm{L}}

\def\W{\mathrm{W}}
\def\re{\mathsf{e}}
\def\sr{\mathsf{r}}

\def\re{\mathsf{e}}
\def\sr{\mathsf{r}}
\def\rB{\mathrm{B}}
\def\rD{\mathrm{D}}

\def\rP{\mathrm{P}}
\def\rp{\mathrm{p}}
\def\rq{\mathrm{q}}
\def\rt{\mathrm{t}}

\def\tF{\mathtt{F}}
\def\tBFD{\mathtt{BFD}}

\def\ttd{\mathtt{d}}

\def\bdiv{\mathbf{div}}

\def\tr{\mathrm{tr}}

\def\div{\mathrm{div}}
\def\dist{\mathrm{dist}\,}

\def\pil{\left<}
\def\pir{\right>}

\def\coeff{\mathbf{coeff}}
\def\iter{\mathtt{iter}}
\def\tol{\textsf{tol}}
\def\DOF{\mathtt{DOF}}

\def\qin{{\quad\hbox{in}\quad}}
\def\qon{{\quad\hbox{on}\quad}}
\def\qan{{\quad\hbox{and}\quad}}

\def\ov{\overline}

\def\wt{\widetilde}
\def\wh{\widehat}

\usepackage[normalem]{ulem}
\newtheorem{thm}{Theorem}[section]

\newtheorem{lem}[thm]{Lemma}

\newenvironment{proof}{\noindent{\it Proof.}}{\hfill$\square$}

\numberwithin{equation}{section}
\numberwithin{figure}{section}
\numberwithin{table}{section}

\allowdisplaybreaks

\title{A mixed FEM for the coupled Brinkman--Forchheimer/Darcy problem}

\author{{\sc Sergio Caucao}\thanks{Departamento de Matem\'atica y F\'isica Aplicadas, 
Universidad Cat\'olica de la Sant\'isima Concepci\'on, Casilla 297, Concepci\'on, Chile, and Grupo de Investigaci\'on en An\'alisis Num\'erico y C\'alculo Cient\'ifico, GIANuC$^2$, Concepci\'on, Chile, 
email: {\tt scaucao@ucsc.cl}. Supported in part by ANID-Chile through the project {\sc Centro de Mode\-lamiento Matem\'atico} (FB210005) and Fondecyt project 11220393.}
\quad
{\sc Marco Discacciati}\thanks{Department of Mathematical Sciences, Loughborough University, Epinal Way, Loughborough LE11 3TU, UK, email: {\tt m.discacciati@lboro.ac.uk}.}}

\date{ }

\begin{document}

\maketitle

\begin{abstract}
\noindent
This paper develops the {\it a priori} analysis of a mixed finite 
element method for the filtration of an incompressible fluid through 
a non-deformable saturated porous medium with heterogeneous permeability. 
Flows are governed by the Brinkman--Forchheimer and Darcy equations
in the more and less permeable regions, respectively,
and the corresponding transmission conditions are given by
mass conservation and continuity of momentum.
We consider the standard mixed formulation in the Brinkman--Forchheimer
domain and the dual-mixed one in the Darcy region,
and we impose the continuity of the normal velocities 
by introducing suitable Lagrange multiplier.  
The finite element discretization involves Bernardi--Raugel and Raviart--Thomas elements for the velocities,
piecewise constants for the pressures, and continuous piecewise linear
elements for the Lagrange multiplier.
Stability, convergence, and {\it a priori} error estimates for the 
associated Galerkin scheme are obtained. 
Numerical tests illustrate the theoretical results.  

\end{abstract}



\maketitle


\section{Introduction}

This paper focuses on the formulation, analysis and numerical approximation of a non-linear coupled problem to model the filtration of an incompressible fluid through a non-deformable saturated porous medium with heterogeneous permeability. More precisely, we consider the case where two different regions are present inside the porous medium domain. In one of them, the permeability and the Reynolds number are low enough to ensure that the classical Darcy's law \cite{Darcy:1856:LFP} provides a valid model to describe the motion of the fluid. In the second region, this is no longer the case as the permeability becomes higher and the effects of frictional forces and inertia cannot be neglected due to the higher flow rate. Therefore, the nonlinear Brinkman--Forchheimer model (see, e.g., \cite{Ehlers:2022:AAM,cy-2021,covy-2022}) must be introduced to accurately represent the fluid flow taking into account the increased flow rates and the effect of viscous forces.  
Similar models have been recently considered in \cite{lcs2021}, and they have been used, e.g., to model fractures \cite{Frih:2008:ComputGeosci} with a porous domain.

This modeling approach gives rise to a global nonlinear coupled model defined in neighboring but non-overlapping regions inside the porous medium. To the best of the authors' knowledge, the well-posedness of this problem has not been studied yet, and suitable discretization techniques must be introduced to guarantee mass conservation throughout the porous medium domain as well as an accurate representation of the fluid velocity and pressure.

These issues are addressed in this paper which is organized as follows. 
In Section \ref{sec:continuous-formulation} we introduce the coupled Brinkman--Forchheimer/Darcy problem and its variational formulation. 
The analysis of the problem at the continuous level is carried out in Section \ref{sec:analysis-continuous-problem}, while Section \ref{sec:galerkin-approximation} focuses on the Galerkin finite element approximation of the coupled problem and on its \textit{a priori} analysis. Two numerical experiments are finally presented in Section \ref{sec:numerical-results} to illustrate the theoretical results.

We conclude this section by introducing some notations that will be used throughout the rest of the paper.
Let $\cO\subset \R^n$, $n\in \{2,3\}$, denote a domain with Lipschitz boundary $\Gamma$. 
For $s\geq 0$ and $\rp\in[1,+\infty]$, we denote by $\L^\rp(\cO)$ and $\W^{s,\rp}(\cO)$ 
the usual Lebesgue and Sobolev spaces endowed with the norms $\|\cdot\|_{0,\rp;\cO}$ and $\|\cdot\|_{s,\rp;\cO}$, respectively.
Note that $\W^{0,\rp}(\cO)=\L^\rp(\cO)$. 
If $\rp = 2$, we write $\H^{s}(\cO)$ instead of $\W^{s,2}(\cO)$, and denote 
the corresponding norm by $\|\cdot\|_{s,\cO}$. We will denote the corresponding vectorial and tensorial counterparts of a generic scalar functional space $\H$ by $\bH$ and $\bbH$. 
The $\L^2(\Gamma)$ inner product or duality pairing
is denoted by $\pil\cdot,\cdot\pir_\Gamma$.
In turn, for any vector field $\bv:=(v_i)_{i=1,n}$, we set the gradient and divergence operators as
\begin{equation*}
\nabla\bv := \left(\frac{\partial\,v_i}{\partial\,x_j}\right)_{i,j=1,n}\qan
\div(\bv) := \sum^{n}_{j=1} \frac{\partial\,v_j}{\partial\,x_j}.
\end{equation*}
When no confusion arises $| \cdot |$ will denote the Euclidean norm in $\R^n$ or $\R^{n\times n}$.
In addition, in the sequel we will make use of the well-known H\"older inequality given by
\begin{equation*}
\int_{\cO} |f\,g| \leq \|f\|_{0,\rp;\cO}\,\|g\|_{0,\rq;\cO}
\quad \forall\, f\in \L^\rp(\cO),\,\forall\, g\in \L^\rq(\cO), 
\quad\mbox{with}\quad \frac{1}{\rp} + \frac{1}{\rq} = 1 \,.
\end{equation*}
Finally, we recall that $\H^1(\cO)$ is continuously embedded into $\L^\rp(\cO)$ for $\rp\geq 1$ 
if $n=2$ or $\rp\in [1,6]$ if $n=3$.
More precisely, we have the following inequality
\begin{equation}\label{eq:Sobolev-inequality}
\|w\|_{0,\rp;\cO} 
\,\leq\, C_{i_\text{p}} 
\|w\|_{1,\cO}\quad 
\forall\,w \in \H^1(\cO), 
\end{equation}
with $C_{i_\text{p}}>0$ a positive constant 
depending only on $|\cO|$ and $\rp$ (see \cite[Theorem 1.3.4]{Quarteroni-Valli}).


\section{Formulation of the model problem}\label{sec:continuous-formulation}

In this section we introduce the model problem at the continuous level and we derive the corresponding weak formulation.
For simplicity of exposition we set the problem in $\R^2$. However, our study can be extended to the $3$D case with few modifications, which we will be pointed out appropriately in the paper.

\subsection{The model problem}

In order to describe the geometry, we let $\Omega_\rB$ and $\Omega_\rD$ be two bounded and simply connected polygonal domains in $\R^2$ such that $\partial\Omega_\rB\cap\partial\Omega_\rD = \Sigma \neq \emptyset$ and $\Omega_\rB\cap\Omega_\rD = \emptyset$. 
Then, let $\Gamma_\rB := \partial\Omega_\rB \setminus \ov{\Sigma}$, $\Gamma_\rD := \partial\Omega_\rD \setminus \ov{\Sigma}$, and denote by $\bn$ the unit normal vector on the boundaries, which is chosen pointing outward from $\Omega := \Omega_\rB\cup\Sigma\cup\Omega_\rD$ and $\Omega_\rB$ (and hence inward to $\Omega_\rD$ when seen on $\Sigma$). 
On $\Sigma$ we also consider a unit tangent vector $\bt$ (see Figure~\ref{fig:dominio-2d}). 
Then, given source terms $\f_\rB$, $\f_\rD$, and $g_\rD$, we are interested in the coupling of the Brinkman--Forchheimer and Darcy equations, which is formulated in terms of the velocity-pressure pair $(\bu_\star,p_\star)$ in $\Omega_\star$, with $\star\in \{\rB,\rD\}$.
More precisely, the sets of equations in the Brinkman--Forchheimer and Darcy domains $\Omega_\rB$ and $\Omega_\rD$, are, respectively,
\begin{equation}\label{eq:BF-model}
\begin{array}{c}
\bsi_\rB = - p_\rB\bbI + \mu\nabla\bu_\rB \qin \Omega_\rB, \quad
\bK^{-1}_\rB\bu_\rB + \tF\,|\bu_{\rB}|^{\rp-2}\bu_{\rB} - \bdiv(\bsi_\rB) = \f_\rB \qin \Omega_\rB, \\ [1ex]
\div(\bu_\rB) = 0 \qin \Omega_\rB, \quad \bu_\rB = \0 \qon \Gamma_\rB,
\end{array}
\end{equation}
and
\begin{equation}\label{eq:Darcy-model}
\bK^{-1}_\rD\bu_\rD + \nabla p_\rD = \f_\rD \qin \Omega_\rD, \quad
\div(\bu_\rD) = g_\rD \qin \Omega_\rD, \quad \bu_\rD\cdot\bn = 0 \qon \Gamma_\rD,
\end{equation}
where $\bsi_\rB$ is the Cauchy stress tensor, $\mu$ is the kinematic viscosity of the fluid,  
$\tF > 0$ is the Forchheimer coefficient, $\rp$ is a given number with $\rp\in [3,4]$, and 
$\bK_{\star}\in \bbL^\infty(\Omega_\star)$ are symmetric
tensors in $\Omega_\star$, with $\star\in \{\rB,\rD\}$, equal to the symmetric permeability tensors scaled by the kinematic viscosity. 
Throughout the paper we assume that there exists $C_\star > 0$ such that 
\begin{equation}\label{eq:permeability-constrain}
\bw\cdot\bK^{-1}_\star(\bx)\bw \geq C_\star |\bw|^2,
\end{equation}
for almost all $\bx\in\Omega_\star$, and for all $\bw\in\R^2$. 
\begin{figure}[H]
\centering\includegraphics[scale=1]{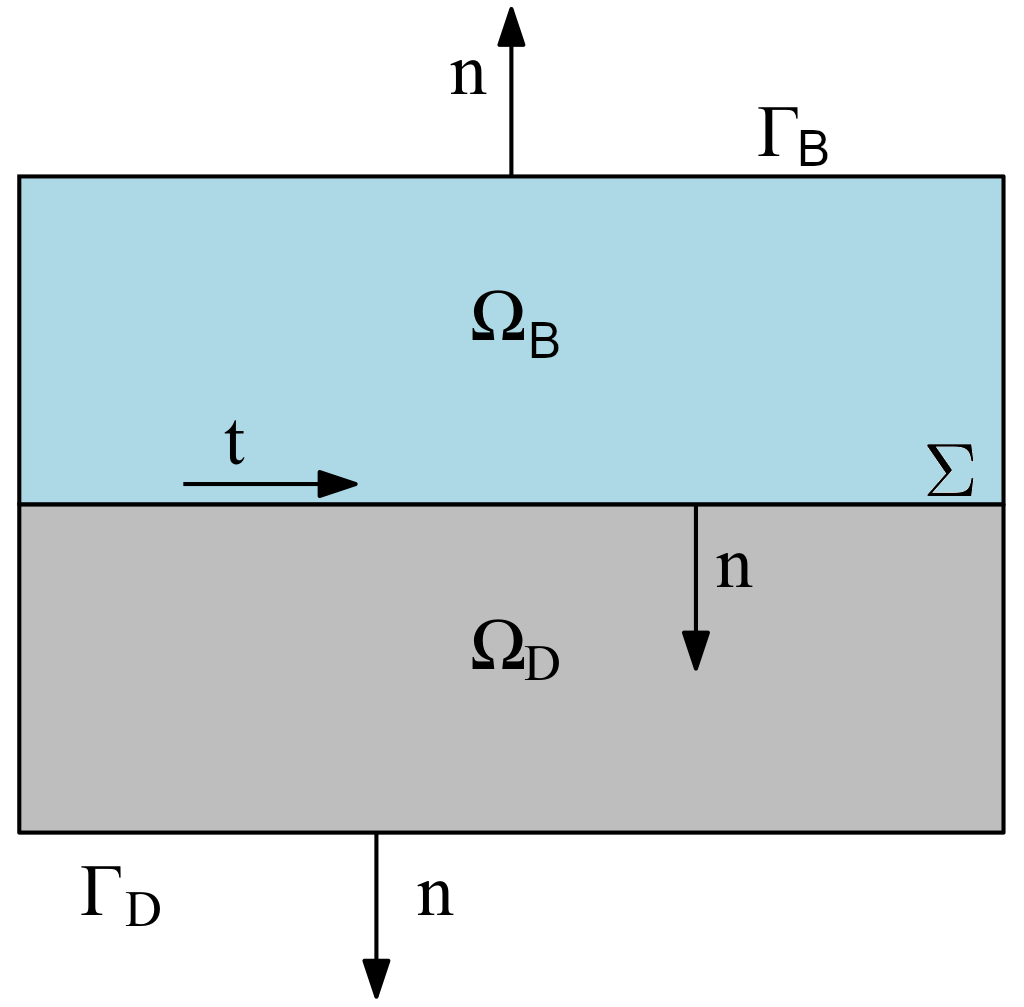}
\caption{Sketch of a 2D geometry of the coupled Brinkman--Forchheimer/Darcy model}
\label{fig:dominio-2d}
\end{figure}
In addition, according to the compressibility conditions, the boundary conditions on $\bu_\rD$ and $\bu_\rB$, and the principle of mass conservation (cf. \eqref{eq:transmission-condition} below), $g_\rD$ must satisfy the compatibility condition:
\begin{equation*}
\int_{\Omega_\rD} g_\rD = 0.
\end{equation*}

To couple the Brinkman--Forchheimer and the Darcy models, we propose transmission conditions that impose both the mass conservation and continuity of momentum across the interface $\Sigma$, following similar approaches in, e.g., \cite{Ehrhardt:2012:PICP,Dumitrache:2012:ATN,lcs2021}. More precisely, we consider
\begin{equation}\label{eq:transmission-condition}
\bu_\rB\cdot\bn = \bu_\rD\cdot\bn 
\qan
\ds \bsi_\rB\bn = -p_\rD\bn \qon \Sigma \,.
\end{equation}

Other boundary conditions can be considered. For example, similarly to \cite{do2017}, one could impose
\begin{equation}\label{eq:alternative-BC}
\begin{array}{c}
\ds \bu_\rB = \0 \qon \Gamma^d_\rB\,,\quad \bsi_\rB\bn = \0 \qon \Gamma^n_\rB\,, \\[2ex]
\ds \bu_\rD\cdot\bn = 0 \qon \Gamma^d_\rD\,,\quad p_\rD = 0 \qon \Gamma^n_\rD\,,
\end{array}
\end{equation}
where $\Gamma^d_\rB\cup \Gamma^n_\rB = \Gamma_\rB, \Gamma^d_\rD\cup \Gamma^n_\rD=\Gamma_\rD$, 
and $\Gamma^d_\rB\cap \Sigma = \emptyset, \Gamma^n_\rD\cap \Sigma = \emptyset$.
The analysis studied in this work can be extended with minor modifications to the case when \eqref{eq:alternative-BC} are used.
However, for the sake of simplicity, we focus on \eqref{eq:BF-model}--\eqref{eq:Darcy-model} for the analysis,
and consider \eqref{eq:alternative-BC} in one of the numerical examples in Section \ref{sec:numerical-results}.


\subsection{Variational formulation}

In this section we proceed analogously to \cite[Section~2]{gmo2009} and derive a weak formulation of the coupled problem given by \eqref{eq:BF-model}, \eqref{eq:Darcy-model}, and \eqref{eq:transmission-condition}. 
Given $\star\in\{\rB,\rD\}$, let
\begin{equation*}
(p,q)_\star := \int_{\Omega_\star} p\,q,\quad 
(\bu,\bv)_\star := \int_{\Omega_\star} \bu\cdot\bv, \qan
(\bsi,\btau)_\star := \int_{\Omega_\star} \bsi:\btau,
\end{equation*}
where, given two arbitrary tensors $\bsi$ and $\btau$, $\bsi:\btau = \tr(\bsi^\rt\btau)=\ds\sum_{i,j=1}^2 \sigma_{ij}\tau_{ij}$.
Furthermore, we consider the Hilbert space
\begin{equation*}
\bH(\div;\Omega_\rD) := \Big\{\bv_\rD\in\bL^2(\Omega_\rD):\quad \div(\bv_\rD) \in\L^2(\Omega_\rD) \Big\},
\end{equation*}
endowed with the norm
\begin{equation*}
\|\bv_\rD\|_{\div;\Omega_\rD} := \Big( \|\bv_\rD\|^2_{0,\Omega_\rD} + \|\div(\bv_\rD)\|^2_{0,\Omega_\rD} \Big)^{1/2},
\end{equation*}
and the following subspaces of $\bH^1(\Omega_\rB)$ and $\bH(\div;\Omega_\rD)$, respectively,
\begin{align*}
\ds \bH^1_{\Gamma_\rB} (\Omega_\rB) & \,:=\, \Big\{\bv_\rB\in\bH^1(\Omega_\rB):\quad \bv_\rB = \0 \qon \Gamma_\rB \Big\}, \\[1ex]
\ds \bH_{\Gamma_\rD} (\div;\Omega_\rD) & \,:=\, \Big\{ \bv_\rD\in\bH(\div;\Omega_\rD):\quad \bv_\rD\cdot\bn = 0 \qon \Gamma_\rD \Big\} \,.
\end{align*}

We now proceed similarly to \cite{do2017,gmo2009} and test the second equation of \eqref{eq:BF-model}
by $\bv_\rB\in\bH^1_{\Gamma_\rB}(\Omega_\rB)$, integrate by parts and utilize the first and second equations of \eqref{eq:BF-model} and \eqref{eq:transmission-condition}, respectively, to obtain
\begin{equation}\label{eq:variational-1}
\mu(\nabla\bu_\rB,\nabla\bv_\rB)_\rB   
+ (\bK^{-1}_\rB\bu_\rB,\bv_\rB)_\rB	
+ \tF\,(|\bu_\rB|^{\rp-2}\bu_\rB,\bv_\rB)_\rB 
- (p_\rB,\div(\bv_\rB))_\rB 
+ \pil\bv_\rB\cdot \bn,\lambda\pir_\Sigma
= (\f_\rB,\bv_\rB)_\rB\,,
\end{equation}
for all $\bv_\rB\in \bH^1_{\Gamma_\rB}(\Omega_\rB)$, where $\lambda$ is a further
unknown representing the trace of the Darcy porous medium pressure on
$\Sigma$, that is $\lambda=p_\rD|_\Gamma \in \H^{1/2}(\Sigma)$. 
Note that, in principle, the space for $p_\rD$ does not allow enough regularity for the trace $\lambda$ to exist. 
However, remark that the solution of \eqref{eq:Darcy-model} has the pressure in $\H^{1}(\Omega_\rD)$.

Then, we incorporate the incompressibility condition in $\Omega_\rB$ weakly
as 
\begin{equation}\label{eq:variational-2}
(q_\rB,\div(\bu_\rB))_\rB=0 \quad \forall\,q_\rB\in \L^2(\Omega_\rB).
\end{equation}
Next, we multiply the first equation of
\eqref{eq:Darcy-model} by $\bv_\rD\in \bH_{\Gamma_\rD} (\div;\Omega_\rD)$ and integrate by parts to obtain
\begin{equation}\label{eq:variational-3}
(\bK^{-1}_\rD\bu_\rD,\bv_\rD)_\rD - (p_\rD,\div(\bv_\rD))_\rD
-\pil\bv_\rD\cdot\bn,\lambda\pir_\Sigma=(\mathbf{f}_\rD,\bv_\rD)_\rD,
\end{equation}
for all $\bv_\rD\in \bH_{\Gamma_\rD} (\div;\Omega_\rD)$.
Finally, we impose the second equation of \eqref{eq:Darcy-model} and
the first equation of \eqref{eq:transmission-condition} weakly as follows
\begin{equation}\label{eq:variational-4}
(q_\rD,\div(\bu_\rD))_\rD=(g_\rD, q_\rD)_\rD \qquad \forall\,q_\rD\in \L^2(\Omega_\rD),
\end{equation}
and
\begin{equation}\label{eq:variational-5}
\pil\bu_\rB\cdot\bn-\bu_\rD\cdot\bn,\xi\pir_\Sigma=0\qquad\forall\,\xi\in \H^{1/2}(\Sigma). 
\end{equation}
As a consequence of the above, we write $\Omega := \Omega_\rB \cup \Sigma \cup \Omega_\rD$,
and define $p:= p_\rB \chi_\rB + p_\rD \chi_\rD$, with $\chi_\star$ being the characteristic function:
\begin{equation*}
\chi_\star := \left\{\begin{array}{lll}
 1 & \mbox{ in } & \Omega_\star, \\ [1ex]
 0 & \mbox{ in } & \Omega\setminus \ov{\Omega}_\star,
\end{array}\right.
\quad\mbox{for }\, \star\in\{\rB,\rD\}\,,
\end{equation*}
to obtain the variational problem: Find $\bu_\rB \in \bH^1_{\Gamma_\rB} (\Omega_\rB),
\,p \in \L^2(\Omega),\,\bu_\rD \in \bH_{\Gamma_\rD} (\div;\Omega_\rD)$ and
$\lambda\in \H^{1/2}(\Sigma)$ such that \eqref{eq:variational-1}--\eqref{eq:variational-5} hold.

Now, let us observe that if $(\bu_\rB,\bu_\rD,p,\lambda)$ is a solution of the variational problem, then
for all $c\in \R$, $(\bu_\rB,\bu_\rD,p+c,\lambda+c)$ is also a solution.
Then, we avoid the non-uniqueness of  \eqref{eq:variational-1}--\eqref{eq:variational-5} by requiring
from now on that $p\in \L_0^2(\Omega)$, where
\begin{equation*}
\L^2_0(\Omega) := \Big\{q\in \L^2(\Omega):\quad \int_\Omega q =0 \Big\}.
\end{equation*}
In this way, we group the spaces and unknowns as follows:
\begin{equation*}
\begin{array}{c}
\bH := \bH^1_{\Gamma_\rB} (\Omega_\rB)\times \bH_{\Gamma_\rD} (\div;\Omega_\rD), \quad 
\bQ := \L^2_0(\Omega)\times \H^{1/2}(\Sigma), \\ [1ex]
\ds \bu := (\bu_\rB,\bu_\rD)  \in \bH, \quad (p,\lambda) \in \bQ,
\end{array}
\end{equation*}
and propose the mixed variational formulation: 
Find $(\bu,(p,\lambda)) \in \bH\times\bQ$, such that
\begin{equation}\label{eq:mixed-variational-formulation}
\begin{array}{llll}
[\ba(\bu),\bv] + [\bb(\bv),(p,\lambda)] & = & [\f,\bv] & \forall\,\bv:=(\bv_\rB,\bv_\rD) \in \bH, \\ [2ex]
[\bb(\bu),(q,\xi)] & = & [\g,(q,\xi)] & \forall\,(q,\xi) \in \bQ,
\end{array}
\end{equation}
where, the operator $\ba : \bH \to \bH'$ is defined by
\begin{equation}\label{eq:definition-a}
[\ba(\bu),\bv] 
\,:=\, \mu (\nabla\bu_\rB,\nabla\bv_\rB)_\rB 
+ (\bK^{-1}_\rB\bu_\rB,\bv_\rB)_\rB
+ \tF\,(|\bu_\rB|^{\rp-2}\bu_\rB,\bv_\rB)_\rB  
+ \left(\bK^{-1}_\rD \bu_\rD,\bv_\rD\right)_\rD\,, 
\end{equation}
whereas the operator $\bb:\bH\to\bQ'$ is given by
\begin{equation}\label{eq:definition-b}
[\bb(\bv),(q,\xi)] := - (q,\div(\bv_\rB))_\rB - (q,\div(\bv_\rD))_\rD + \pil \bv_\rB\cdot\bn - \bv_\rD\cdot\bn,\xi \pir_\Sigma.
\end{equation}
In turn, the functionals $\f\in \bH'$ and $\g\in \bQ'$ are defined by
\begin{equation}\label{eq:definition-rhs}
[\f,\bv] := (\f_\rB,\bv_\rB)_\rB + (\f_\rD,\bv_\rD)_\rD \qan [\g,(q,\xi)] := -(g_\rD,q)_\rD.
\end{equation}
In all the terms above, $[\,\cdot,\cdot\,]$ denotes the duality pairing induced by the corresponding operators.


\section{Analysis of the continuous coupled problem}\label{sec:analysis-continuous-problem}

In this section we establish the solvability of \eqref{eq:mixed-variational-formulation}.
We first collect some preliminaries results that will be used in the forthcoming analysis

\subsection{Preliminary results}

We begin by recalling the following abstract result \cite[Theorem 3.1]{cgo2021}, which in turn, is a modification of \cite[Theorem 3.1]{cdgo2020}.
\begin{thm}\label{thm:existence-approximation-nonlinear-problems}
Let $X_{1}$, $X_{2}$  and $Y$ be separable and reflexive Banach spaces, 
being $X_{1}$ and $X_{2}$ uniformly convex, and set $X:=X_{1}\times X_{2}$. 
Let $\cA:X\to X'$ be a nonlinear operator, $\mathcal{B}\in \mathcal{L}(X, Y')$, and let 
$V$ be the kernel of $\cB$, that is,
\begin{equation*}
V \,:=\, \Big\{ v=(v_1,v_2)\in X :\quad \cB (v) \,=\, \0 \Big\}\,.
\end{equation*}
Assume that
\begin{enumerate}
\item[\rm{(i)}] there exist constants $L > 0$ and $p_{1},p_{2}\geq 2$, such that
\begin{equation*}
\|\cA(u) - \cA(v)\|_{X'} 
\,\leq\, L\,\sum_{j=1}^{2}\Big\{ \|u_{j} - v_{j}\|_{X_{j}}
+ \big(\|u_{j}\|_{X_{j}} + \|v_{j}\|_{X_{j}}\big)^{p_{j}-2}\|u_{j} - v_{j}\|_{X_{j}} \Big\}
\end{equation*}
for all $u=(u_{1},u_{2}), v=(v_{1},v_{2})\in X$,
		
\item[\rm{(ii)}] 
the family of operators $\Big\{ \cA(\,\cdot+ z): V\to V':\quad
z\in X\Big\}$ is  uniformly  strongly  monotone, that is 
there exists $\alpha>0$ such that
\begin{equation*}
[\cA(u + z) - \cA(v + z),u - v] 
\,\geq\, \alpha\,\|u - v\|_{X}^{2}\,,
\end{equation*}
for all $z\in X$, and for all $u,v\in V$, and
		
\item[\rm{(iii)}] 
there exists $\beta>0$ such that
\begin{equation*}
\sup_{\substack{v\in X \\ v\neq 0}}\frac{[\cB(v),\tau]}{\|v\|_{X}}
\,\geq\, \beta\,\|\tau\|_{Y}\quad \forall\,\tau\in Y\,.
\end{equation*}
\end{enumerate}
Then, for each $(\mathcal{F},\mathcal{G})\in X'\times Y'$ 
there exists a unique $(u,\sigma)\in X\times Y$ such that
\begin{equation}\label{eq:teo-exists-unique-nolinear}
\begin{array}{rcll}
[\cA(u),v] + [\cB(v),\sigma] & = & [\mathcal{F},v] & \forall\,v\in X\,, \\ [1ex]
[\cB(u),\tau] & = & [\mathcal{G},\tau] & \forall\,\tau\in Y\,.
\end{array}
\end{equation}
Moreover, there exist positive constants $C_{1}$ and $C_{2}$, 
depending only on $L, \alpha$, and $\beta$, such that
\begin{equation}\label{eq:a-priori-u-thm-nonlinear-problems}
\|u\|_{X} 
\,\leq\, C_{1}\,\mathcal{M}(\mathcal{F},\mathcal{G})
\end{equation}
and
\begin{equation}\label{eq:a-priori-sigma-thm-nonlinear-problems}
\|\sigma\|_{Y} 
\,\leq\, C_{2}\,\bigg\{ \mathcal{M}(\mathcal{F},\mathcal{G}) +\sum^2_{j=1} \mathcal{M}(\mathcal{F},\mathcal{G})^{p_j-1} \bigg\} \,,
\end{equation}
where
\begin{equation}\label{def-M-F-G}
\mathcal{M}(\mathcal{F},\mathcal{G}) 
\,:=\, \|\mathcal{F}\|_{X'} + \|\mathcal{G}\|_{Y'} + \sum^2_{j=1} \|\mathcal{G}\|_{Y'}^{p_{j}-1} + \|\mathcal{A}(0)\|_{X'} \,.
\end{equation}
\end{thm}

Next, we establish the stability properties of the operators and functionals 
involved in \eqref{eq:mixed-variational-formulation}. We begin by observing that 
the operator $\bb$ and functionals $\f$ and $\g$ are linear.
In turn, from \eqref{eq:definition-b} and \eqref{eq:definition-rhs}, and
employing the Cauchy--Schwarz inequality, there exists a positive constant $C_\bb$, such that
\begin{equation}\label{eq:continuity-b}
\big| [\bb(\bv),(q,\xi)] \big| \,\leq\, C_{\bb}\,\|\bv\|_{\bH}\,\|(q,\xi)\|_{\bQ}
\end{equation}
and
\begin{align}
\big|[\f,\bv]\big|
&\,\leq\, \Big\{ \|\f_{\rB}\|_{0,\Omega_{\rB}} + \|\f_{\rD}\|_{0,\Omega_{\rD}} \Big\} \|\bv\|_{\bH} \quad
\forall\,\bv\in \bH, \label{eq:stability-f} \\[2ex]
\big|[\g,(q,\xi)]\big|
&\,\leq\, \|g_\rD\|_{0,\Omega_{\rD}}\,\|(q,\xi)\|_{\bQ} \quad
\forall\, (q,\xi)\in \bQ\,, \label{eq:stability-g} 
\end{align}
which implies that $\bb, \f$ and $\g$ are bounded and continuous.
In addition, employing the Cauchy--Schwarz and H\"older inequalities,
and the continuous injection $\bi_\rp$ of $\bH^1(\Omega_\rB)$ into $\bL^{\rp}(\Omega_\rB)$, 
with $\rp\in [3,4]$ (cf. \eqref{eq:Sobolev-inequality}), it is readily seen that the nonlinear 
operator $\ba$ in \eqref{eq:definition-a} is bounded, that is
\begin{equation}\label{eq:a-bounded}
\big| [\ba(\bu),\bv] \big| \,\leq\, 
C_{\ba}\,\Big\{ \|\bu_{\rB}\|_{1,\Omega_\rB} + \|\bu_{\rB}\|^{\rp-1}_{1,\Omega_{\rB}} 
+\|\bu_\rD\|_{\div;\Omega_\rD} \Big\} \|\bv\|_{\bH}\,,
\end{equation}
with $C_{\ba}>0$ depending on $\mu, \tF, C_{\bi_\rp}, \bK_\rB$, and $\bK_\rD$.

Finally, we follow \cite{gos2012} to recall some preliminary results concerning boundary conditions and extension operators. 
Given $\bv_\rD\in\bH_{\Gamma_\rD}(\div;\Omega_\rD)$, the boundary condition
$\bv_\rD\cdot\bn = 0$ on $\Gamma_\rD$ means (see, e.g., \cite[Appendix A]{ejs2009} and \cite{gos2012,do2017})
\begin{equation*}
\pil\bv_\rD\cdot\bn, E_{0,\rD}(\xi) \pir_{\partial\Omega_\rD} \,=\, 0 \quad \forall\,\xi\in 
\H^{1/2}_{00}(\Gamma_\rD) \,,
\end{equation*}
where $\pil \cdot, \cdot\pir_{\partial\Omega_\rD}$ stands for the usual duality pairing between
$\H^{-1/2}(\partial\Omega_\rD)$ and $\H^{1/2}(\partial\Omega_\rD)$ with respect to the
$\L^2(\partial\Omega_\rD)$-inner product, $E_{0,\rD}: \H^{1/2}(\Gamma_\rD)\to \L^2(\partial\Omega_\rD)$ is the extension operator defined by
\begin{equation*}
E_{0,\rD}(\xi):=\left\{
\begin{array}{lll}
\xi&{\rm on}&\Gamma_\rD \\[0.5ex]
0&{\rm on}&\Sigma
\end{array}
\right.\quad\forall\,\xi\in \H^{1/2}(\Gamma_\rD) \,,
\end{equation*}
and
$\H^{1/2}_{00}(\Gamma_\rD) = \big\{ \xi\in \H^{1/2}(\Gamma_\rD) :\,\, E_{0,\rD}(\xi)\in \H^{1/2}(\partial\Omega_\rD) \big\}$,
endowed with the norm $\|\xi\|_{1/2,00;\Gamma_\rD} := \|E_{0,\rD}(\xi)\|_{1/2,\partial\Omega_\rD}$.

As a consequence, it is not difficult to prove (see, e.g., \cite[Section 2]{gs2007})
that the restriction of $\bv_\rD\cdot\bn$ to $\Sigma$ can be identified with an element of $\H^{-1/2}(\Sigma)$, namely
\begin{equation}\label{eq:vD-to-vD-ED}
\pil \bv_\rD\cdot\bn, \xi\pir_\Sigma \,:=\, \pil \bv_\rD\cdot\bn, E_\rD(\xi)\pir_{\partial\Omega_\rD} 
\quad \forall\,\xi\in \H^{1/2}(\Sigma)\,,
\end{equation}
where $E_\rD:\H^{1/2}(\Sigma)\to \H^{1/2}(\partial\Omega_\rD)$ is any bounded extension operator.
In particular, given $\xi\in \H^{1/2}(\Sigma)$, one could define $E_\rD(\xi) := z|_{\partial\Omega_\rD}$,
where $z\in \H^1(\Omega_\rD)$ is the unique solution of the boundary value problem: $\Delta z = 0$ in $\Omega_\rD$, $z=\xi$ on $\Sigma$, $\nabla z\cdot\bn = 0$ on $\Gamma_\rD$.
In addition, one can show (see \cite[Lemma 2.2]{gs2007}) that for all $\psi\in \H^{1/2}(\partial\Omega_\rD)$, there exist unique elements $\psi_\Sigma\in \H^{1/2}(\Sigma)$ and $\psi_{\Gamma_\rD}\in \H^{1/2}_{00}(\Gamma_\rD)$ such that
\begin{equation}\label{eq:psi-decomposition}
\psi \,=\, E_\rD(\psi_\Sigma) \,+\, E_{0,\rD}(\psi_{\Gamma_\rD}) \,,
\end{equation}
and there exist $C_1, C_2>0$, such that
\begin{equation*}
C_1 \Big\{ \|\psi_\Sigma\|_{1/2,\Sigma} + \|\psi_{\Gamma_\rD}\|_{1/2,00;\Gamma_\rD} \Big\} 
\,\leq\, \|\psi\|_{1/2,\partial\Omega_\rD} 
\,\leq\, C_2 \Big\{ \|\psi_\Sigma\|_{1/2,\Sigma} + \|\psi_{\Gamma_\rD}\|_{1/2,00;\Gamma_\rD} \Big\}\,.
\end{equation*}

\subsection{Existence and uniqueness of solution}

We begin by observing that the problem \eqref{eq:mixed-variational-formulation} has the same structure as \eqref{eq:teo-exists-unique-nolinear}.
Therefore, in order to apply Theorem \ref{thm:existence-approximation-nonlinear-problems}, we notice that, thanks to the uniform convexity and
separability of $\L^\rp(\Omega)$ for $\rp\in (1,+\infty)$, all the spaces involved in \eqref{eq:mixed-variational-formulation}, that is, $\bH^1_{\Gamma_\rB}(\Omega_\rB)$, $\bH_{\Gamma_\rD}(\div;\Omega_\rD)$, $\L^2_0(\Omega)$, and $\H^{1/2}(\Sigma)$, share the same properties, which implies that $\bH$ and $\bQ$ are uniformly convex and separable as well.

We continue our analysis by proving that the nonlinear operator $\ba$ satisfies hypothesis (i) of Theorem \ref{thm:existence-approximation-nonlinear-problems} with $\rp_1=\rp\in [3,4]$ and $\rp_2=2$.

\begin{lem}\label{lem:continuity-a}
Let $\rp\in [3,4]$. Then, there exists $L_{\tBFD}>0$, depending on $\tF, \bK_{\rD}, \bK_{\rB}$, and $C_{\bi_\rp}$, such that
\begin{equation}\label{eq:continuity-a}
\begin{array}{l}
\ds \|\ba(\bu) - \ba(\bv)\|_{\bH'} \\[2ex]
\ds\quad \leq\, L_{\tBFD}\,\Big\{ \|\bu_\rB - \bv_\rB\|_{1,\Omega} + \|\bu_\rD - \bv_\rD\|_{\div;\Omega} + \big(\|\bu_{\rB}\|_{1,\Omega} + \|\bv_{\rB}\|_{1,\Omega}\big)^{\rp-2} \|\bu_\rB - \bv_\rB\|_{1,\Omega} \Big\} \,,
\end{array}
\end{equation}
for all $\bu=(\bu_{\rB},\bu_{\rD}), \bv=(\bv_{\rB},\bv_{\rD})\in \bH$.
\end{lem}
\begin{proof}
Let $\bu=(\bu_{\rB},\bu_{\rD}), \bv=(\bv_{\rB},\bv_{\rD})\in \bH$.
Then, according to the definition of the operator $\ba$ (cf. \eqref{eq:definition-a}),
similarly to the boundedness estimate \eqref{eq:a-bounded}, using H\"older's and 
Cauchy--Schwarz inequalities, we find that
\begin{equation}\label{eq:continuity-1}
\begin{array}{l}
\ds \|\ba(\bu) - \ba(\bv)\|_{\bH'} 
\,\leq\, \tF\, C_{\bi_\rp} \,\||\bu_\rB|^{\rp-2}\bu_\rB - |\bv_\rB|^{\rp-2}\bv_\rB\|_{0,\rq;\Omega_\rB} \\[2ex] 
\ds\quad +\,\, \max\big\{ \mu,\|\bK^{-1}_\rB\|_{0,\infty;\Omega_\rB}, \|\bK^{-1}_\rD\|_{0,\infty;\Omega_\rD} \big\} 
\Big\{ \|\bu_\rB - \bv_\rB\|_{1,\Omega_\rB} + \|\bu_\rD - \bv_\rD\|_{\div;\Omega_\rD} \Big\} \,.
\end{array}
\end{equation}
In turn, applying \cite[Lemma~2.1, eq. (2.1a)]{bl1993} to bound the first term on 
the right hand side of \eqref{eq:continuity-1} and the continuous injection $\bi_\rp$ 
of $\bH^1(\Omega_\rB)$ into $\bL^{\rp}(\Omega_\rB)$, with $\rp\in [3,4]$ 
(cf. \eqref{eq:Sobolev-inequality}), we deduce that there exists $c_\rp > 0$, 
depending only on $|\Omega|$ and $\rp$ such that
\begin{equation}\label{eq:continuity-bound-nonlinear-term}
\begin{array}{c}
\ds \||\bu_\rB|^{\rp-2}\bu_\rB - |\bv_\rB|^{\rp-2}\bv_\rB\|_{0,\rq;\Omega_\rB}   
\,\leq\, c_\rp\,\big(\|\bu_\rB\|_{0,\rp;\Omega_\rB} + \|\bv_\rB\|_{0,\rp;\Omega_\rB}\big)^{\rp-2} \|\bu_\rB - \bv_\rB\|_{0,\rp;\Omega_\rB} \\[2ex]
\ds \,\leq\, c_\rp\, (C_{\bi_\rp})^{\rp-1}\big(\|\bu_\rB\|_{1,\Omega_\rB} + \|\bv_\rB\|_{1,\Omega_\rB}\big)^{\rp-2} \|\bu_\rB - \bv_\rB\|_{1,\Omega_\rB} \,.
\end{array}
\end{equation}
Thus, replacing \eqref{eq:continuity-bound-nonlinear-term} back into \eqref{eq:continuity-1},
we obtain \eqref{eq:continuity-a} with 
$$L_{\tBFD} = \max\big\{ \mu, \|\bK^{-1}_\rB\|_{0,\infty;\Omega_\rB}, \|\bK^{-1}_\rD\|_{0,\infty;\Omega_\rD}, \tF\,c_\rp\,(C_{\bi_\rp})^{\rp} \big\}\,,$$ 
which completes the proof.
\end{proof}

Now, let us look at the kernel of the operator $\bb$, that is
\begin{equation}\label{eq:kernel-V}
\bV \,:=\, \Big\{ \bv\in \bH :\quad [\bb(\bv),(q,\xi)] = 0 \quad \forall\,(q,\xi)\in \bQ \Big\}\,.
\end{equation}
According to the definition of $\bb$ (cf. \eqref{eq:definition-b}), we observe that 
$\bv=(\bv_\rB, \bv_\rD)\in \bV$ if and only if
\begin{equation*}
(\div(\bv_\rB),q)_\rB + (\div(\bv_\rD),q)_\rD \,=\, 0 \quad \forall\,q\in \L^2_0(\Omega)
\end{equation*}
and
\begin{equation*}
\pil \bv_\rB\cdot\bn - \bv_\rD\cdot\bn,\xi \pir_\Sigma \,=\, 0 \quad \forall\,\xi\in \H^{1/2}(\Sigma)\,.
\end{equation*}
In this way, noting that $\L^2(\Omega) = \L^2_0(\Omega) \oplus \R$, and taking $\xi \in \R$ in the latter equation, we deduce that
\begin{equation*}
(\div(\bv_\rB),q)_\rB + (\div(\bv_\rD),q)_\rD \,=\, 0 \quad \forall\,q\in \L^2(\Omega),
\end{equation*}
which implies
\begin{equation}\label{eq:div-uB-uD-0}
\div(\bv_\rB) = 0 \qin \Omega_\rB \qan \div(\bv_\rD) = 0 \qin \Omega_\rD.
\end{equation}

In the following result we show that the operator $\ba$ satisfies hypothesis (ii) of Theorem \ref{thm:existence-approximation-nonlinear-problems} with $\rp_1=\rp\in [3,4]$ and $\rp_2=2$.
\begin{lem}\label{lem:strong-monotonicity-a}
Let $\rp\in [3,4]$. The family of operators $\Big\{ \ba(\cdot + \bz) : \bV\to \bV' :\quad \bz\in \bH \Big\}$ is uniformly strongly monotone, that is, there exists $\gamma_{\tBFD}>0$, such that
\begin{equation}\label{eq:strict-monotonicity-a}
[\ba(\bu + \bz) - \ba(\bv + \bz),\bu - \bv] \,\geq\, \gamma_{\tBFD}\,\|\bu - \bv\|^2_{\bH} \,,
\end{equation}
for all $\bz=(\bz_{\rB},\bz_{\rD})\in \bH$, and for all $\bu=(\bu_{\rB},\bu_{\rD}), \bv=(\bv_{\rB},\bv_{\rD})\in \bV$.
\end{lem}
\begin{proof}
Let $\bz=(\bz_{\rB},\bz_{\rD})\in \bH$, and $\bu=(\bu_{\rB},\bu_{\rD}), \bv=(\bv_{\rB},\bv_{\rD})\in \bV$. 
Then, from the definition of the operator $\ba$ (cf. \eqref{eq:definition-a}), 
the estimate \eqref{eq:permeability-constrain}, and the fact that 
$\div(\bu_\rD - \bv_\rD) = 0$ in $\Omega_\rD$ (cf. \eqref{eq:div-uB-uD-0}), we get
\begin{equation}\label{eq:strong-monotonicity-1}
\begin{array}{c}
\ds \big[\ba(\bu + \bz) - \ba(\bv + \bz),\bu - \bv\big] 
\,\geq\, \min\{ \mu, C_{\bK_\rB} \}\|\bu_\rB - \bv_\rB\|^2_{1,\Omega_\rB}
+ C_{\bK_\rD}\,\|\bu_\rB - \bv_\rB\|^2_{\div;\Omega_\rD} \\ [2ex]
\ds\quad  
+\,\, \tF\,(|\bu_\rB+\bz_\rB|^{\rp-2}(\bu_\rB+\bz_\rB) - |\bv_\rB+\bz_\rB|^{\rp-2}(\bv_\rB+\bz_\rB),\bu_\rB - \bv_\rB)_{\rB} \,.
\end{array}
\end{equation}
In turn, employing \cite[Lemma~2.1, eq. (2.1b)]{bl1993} to bound the last term in \eqref{eq:strong-monotonicity-1}, we deduce that there exists $C_{\rp}>0$ depending only on $|\Omega|$ and $\rp$ such that
\begin{equation*}
(|\bu_\rB+\bz_\rB|^{\rp-2}(\bu_\rB+\bz_\rB) - |\bv_\rB+\bz_\rB|^{\rp-2}(\bv_\rB+\bz_\rB),\bu_\rB - \bv_\rB)_{\rB}
\,\geq\, C_\rp\,\|\bu_\rB - \bv_\rB\|^\rp_{0,\rp;\Omega_\rB}
\,\geq\, 0 \,,
\end{equation*}
which, together with \eqref{eq:strong-monotonicity-1}, implies
\begin{equation}\label{eq:strong-monotonicity-2}
\big[\ba(\bu + \bz) - \ba(\bv + \bz),\bu - \bv\big] 
\,\geq\, \min\{ \mu, C_{\bK_\rB} \}\|\bu_\rB - \bv_\rB\|^2_{1,\Omega_\rB}
+ C_{\bK_\rD}\,\|\bu_\rB - \bv_\rB\|^2_{\div;\Omega_\rD} \,.
\end{equation}
Hence, it is clear that \eqref{eq:strong-monotonicity-2} yields \eqref{eq:strict-monotonicity-a}, 
with $\gamma_{\tBFD} = \min\big\{ \mu, C_{\bK_\rB}, C_{\bK_\rD} \big\}$, concluding the proof.
\end{proof}

We end the verification of the hypotheses of Theorem \ref{thm:existence-approximation-nonlinear-problems}, with the corresponding inf-sup condition for the operator $\bb$ (cf. \eqref{eq:definition-b}).
The corresponding proof can be found in \cite[Lemma 1]{do2017}.
We just remark that the main tools employed are the extension property \eqref{eq:vD-to-vD-ED} 
and the decomposition \eqref{eq:psi-decomposition}.
Thus, we simply state the result as follows.
\begin{lem}\label{lem:inf-sup-operator-b}
There exists $\beta>0$ such that
\begin{equation}\label{eq:inf-sup-b}
\sup_{\0\neq \bv\in \bH} 
\frac{[b(\bv),(q,\xi)]}{\|\bv\|_{\bH}}
\,\geq\, \beta\,\|(q,\xi)\|_{\bQ} \quad
\forall\,(q,\xi)\in \bQ \,.
\end{equation}
\end{lem}

Now, we are in a position of establishing the well-posedness of problem \eqref{eq:mixed-variational-formulation}.
\begin{thm}\label{thm:well-posedness-result}
Let $\rp\in [3,4]$.
Then, the problem \eqref{eq:mixed-variational-formulation} has a unique solution $(\bu,(p,\lambda))\in \bH\times \bQ$.
Moreover, there exist constants $C_1, C_2>0$, independent of the solution, such that
\begin{equation}\label{eq:a-priori-bound-u}
\|\bu\|_{\bH} \,\leq\, C_1\,\Big( \|\f_\rB\|_{0,\Omega_\rB} + \|\f_\rD\|_{0,\Omega_\rD} 
+ \|g_\rD\|_{0,\Omega_\rD} + \|g_\rD\|^{\rp-1}_{0,\Omega_\rD} \Big) 
\end{equation}
and
\begin{equation}\label{eq:a-priori-bound-p-lam}
\|(p,\lambda)\|_{\bQ} \,\leq\, C_2\,\sum_{j\in \{\rp,2\} } \Big( \|\f_\rB\|_{0,\Omega_\rB} + \|\f_\rD\|_{0,\Omega_\rD} 
+ \|g_\rD\|_{0,\Omega_\rD} + \|g_\rD\|^{\rp-1}_{0,\Omega_\rD} \Big)^{j-1}  \,.
\end{equation}
\end{thm}
\begin{proof}
First, we recall that, from \eqref{eq:continuity-b}, \eqref{eq:stability-f}, and \eqref{eq:stability-g}, $\bb$, $\f$, and $\g$ are all linear and bounded.
Thus, bearing in mind Lemmas~\ref{lem:continuity-a} and \ref{lem:strong-monotonicity-a}, 
and the inf-sup condition of $\bb$ given by \eqref{eq:inf-sup-b} (cf. Lemma \ref{lem:inf-sup-operator-b}), 
a straightforward application of Theorem \ref{thm:existence-approximation-nonlinear-problems}, with $\rp_1=\rp\in [3,4]$ and $\rp_2=2$, to problem \eqref{eq:mixed-variational-formulation} completes the proof.
In particular, noting from \eqref{eq:definition-a} that $\ba(\0)$ is the null functional, we
get from \eqref{def-M-F-G} that 
\begin{equation*}
\cM(\f,\g) = \|\f\|_{\bH'} + 2\,\|\g\|_{\bQ'} + \|\g\|^{\rp-1}_{\bQ'} \,,
\end{equation*}
and hence the {\it a priori} estimates \eqref{eq:a-priori-u-thm-nonlinear-problems} and \eqref{eq:a-priori-sigma-thm-nonlinear-problems} yield
\begin{equation*}
\|\bu\|_{\bH}  \,\leq\, c_1\Big( \|\f\|_{\bH'} + \|\g\|_{\bQ'} + \|\g\|^{\rp-1}_{\bQ'} \Big)
\end{equation*}
and
\begin{equation*}
\|(p,\lambda)\|_{\bQ} \,\leq\, c_2\,\sum_{j\in \{\rp,2\} } \Big( \|\f\|_{\bH'} 
+ \|\g\|_{\bQ'} + \|\g\|^{\rp-1}_{\bQ'} \Big)^{j-1}  \,,
\end{equation*}
with positive constants $c_1, c_2$ depending only on $L_{\tBFD}, \gamma_{\tBFD}$, and $\beta$.
The foregoing inequalities together with the bounds \eqref{eq:stability-f} and \eqref{eq:stability-g} of $\|\f\|_{\bH'}$ and $\|\g\|_{\bQ'}$ imply \eqref{eq:a-priori-bound-u} and \eqref{eq:a-priori-bound-p-lam}, thus completing the proof.
\end{proof}


\section{The Galerkin scheme}\label{sec:galerkin-approximation}

In this section we introduce and analyze the Galerkin scheme of 
problem \eqref{eq:mixed-variational-formulation}.
We analyze its solvability by employing the strategy developed in 
Section \ref{sec:analysis-continuous-problem}.
Finally, we derive the error estimates and obtain the corresponding rates of convergence.

\subsection{Discrete coupled problem}\label{sec:discrete-setting}

Let $\cT^\rB_h$ and $\cT^\rD_h$ be respective triangulations of the domains $\Omega_\rB$ 
and $\Omega_\rD$ formed by shape-regular triangles, denote by $h_\rB$ 
and $h_\rD$ their corresponding mesh sizes, and let $h : = \max\big\{ h_\rB, h_\rD \big\}$. 
Assume that $\cT^\rB_h$ and $\cT^\rD_h$ match on $\Sigma$ so that $\cT_h := \cT^\rB_h \cup \cT^\rD_h$ is a triangulation of 
$\Omega := \Omega_\rB \cup \Sigma \cup \Omega_\rD$. 
Then, given an integer $l\geq 0$ and a subset $S$ of $\R^2$, we denote by $\rP_l(S)$ the space of polynomials of total degree at most $l$ defined on $S$.
For each $T\in \cT^\rD_h$ we consider the local Raviart--Thomas space of the lowest order \cite{Raviart-Thomas}:
\begin{equation*}
\bRT_0(T) := [\rP_0(T)]^2 \oplus \rP_0(T)\,\bx \,,
\end{equation*}
where $\bx:=(x_1,x_2)^\rt$ is a generic vector of $\R^2$.
In addition, for each $T\in \cT^\rB_h$ we denote by $\bBR(T)$ the local Bernardi--Raugel space \cite{br1985}:
\begin{equation*}
\bBR(T) := [\rP_1(T)]^2 \oplus \mathrm{span}\Big\{ \eta_2\eta_3\bn_1,\eta_1\eta_3\bn_2,\eta_1\eta_2\bn_3 \Big\},
\end{equation*}
where $\big\{ \eta_1, \eta_2, \eta_3\big\}$ are the baricentric coordinates of $T$, and $\big\{\bn_1, \bn_2, \bn_3\big\}$ are the unit outward normals to the opposite sides of the corresponding vertices of $T$. Hence, we define the following finite element subspaces:
\begin{align*}
\bH_h (\Omega_\rB)  & \,:= \, \Big\{ \bv\in\bH^1(\Omega_\rB) : \quad \bv|_T \in \bBR(T),\quad \forall\, T\in\cT^\rB_h \Big\} \,, \\
\bH_h(\Omega_\rD)  & \,:=\, \Big\{ \bv\in \bH(\div;\Omega_\rD) : \quad \bv|_T \in \bRT_0(T),\quad \forall\, T\in\cT^\rD_h \Big\} \,, \\
\L_h(\Omega) & \,:=\, \Big\{ q\in \L^2(\Omega) :\quad q|_T\in \rP_0(T), \quad \forall\, T\in \cT_h \Big\} \,.
\end{align*}
Then, the finite element subspaces for the velocities and pressure are, respectively,
\begin{align}
\bH_{h,\Gamma_\rB}(\Omega_\rB) & \,:=\, \bH_h(\Omega_\rB)\cap \bH^1_{\Gamma_\rB}(\Omega_\rB) \,,  \nonumber \\[1ex] 
\bH_{h,\Gamma_\rD}(\Omega_\rD) & \,:=\, \bH_h(\Omega_\rD)\cap \bH_{\Gamma_\rD}(\div;\Omega_\rD) \,, \label{eq:FEM-1} \\[1ex]
\L_{h,0}(\Omega) & \,:=\, \L_h(\Omega)\cap \L^2_0(\Omega) \,. \nonumber
\end{align}

Next, to introduce the finite element subspace of $\H^{1/2}(\Sigma)$, 
we denote by $\Sigma_h$ the partition of $\Sigma$ inherited from $\cT^\rD_h$ (or $\cT^\rB_h$)
and assume without loss of generality, that the number of edges of $\Sigma_h$ is even.
Then, since $\Sigma_{h}$ is inherited from the interior triangulations, it is 
automatically of bounded variation, i.e., the ratio of lengths of adjacent edges is bounded, and so is $\Sigma_{2h}$.
If the number of edges of $\Sigma_h$ is odd, we simply reduce it to the even case by joining
any pair of two adjacent elements, and then construct $\Sigma_{2h}$ from this reduced partition.
Then, we define the following finite element subspace for $\lambda\in \H^{1/2}(\Sigma)$
\begin{equation}\label{eq:FEM-2}
\Lambda_h(\Sigma) := \Big\{ \xi_h\in \cC(\Sigma) :\quad \xi_h|_e\in \rP_1(e) \quad \forall\,e\in \Sigma_{2h} \Big\} \,.
\end{equation}

In this way, grouping the unknowns and spaces as follows:
\begin{equation*}
\begin{array}{c}
\bH_h := \bH_{h,\Gamma_\rB}(\Omega_\rB)\times \bH_{h,\Gamma_\rD} (\Omega_\rD), \quad \bQ_h := \L_{h,0}(\Omega)\times \Lambda_h(\Sigma), \\ [1ex]
\ds \bu_h := (\bu_{\rB,h}, \bu_{\rD,h})  \in \bH_h, \quad (p_h, \lambda_h) \in \bQ_h,
\end{array}
\end{equation*}
where $p_h:= p_{\rB,h} \chi_\rB + p_{\rD,h} \chi_\rD$, the Galerkin scheme for \eqref{eq:mixed-variational-formulation} reads:  
Find $(\bu_h,(p_h,\lambda_h)) \in \bH_h\times\bQ_h$, such that
\begin{equation}\label{eq:discrete-mixed-formulation}
\begin{array}{llll}
[\ba(\bu_h),\bv_h] + [\bb(\bv_h),(p_h,\lambda_h)] & = & [\f,\bv_h] & \forall\,\bv_h:=(\bv_{\rB,h},\bv_{\rD,h}) \in \bH_h \,, \\ [2ex]
[\bb(\bu_h),(q_h,\xi_h)] & = & [\g,(q_h,\xi_h)] & \forall\,(q_h,\xi_h) \in \bQ_h\,.
\end{array}
\end{equation}


Now, let $\Pi_\rB : \bH^1_{\Gamma_\rB}(\Omega_\rB)\to \bH_{h,\Gamma_\rB}(\Omega_\rB)$ 
be the Bernardi--Raugel interpolation operator \cite{br1985}, which is linear and bounded 
with respect to the $\bH^1(\Omega_\rB)$-norm. 
In this regard, we recall that, given $\bv\in\bH^1_{\Gamma_\rB}(\Omega_\rB)$, there holds
\begin{equation}\label{eq:PiB-edge}
\int_e \Pi_\rB(\bv)\cdot\bn \,=\, \int_e \bv\cdot\bn \quad \mbox{for each edge $e$ of } \cT^\rB_h, 
\end{equation}
and hence
\begin{equation}\label{eq:PiB-div}
(\div(\Pi_\rB(\bv)),q_h)_\rB \,=\, (\div(\bv),q_h)_\rB \quad \forall\,q_h\in\L_h(\Omega).
\end{equation}
Equivalently, if $\cP_\rB$ denotes the $\L^2(\Omega_\rB)$-orthogonal projection onto the restriction 
of $\L_h(\Omega)$ to $\Omega_\rB$, then the relation \eqref{eq:PiB-div} can be written as
\begin{equation}\label{eq:PB-div}
\cP_\rB(\div(\Pi_\rB(\bv))) \,=\, \cP_\rB(\div(\bv)) \quad \forall\,\bv\in\bH^1_{\Gamma_\rB}(\Omega_\rB) \,.
\end{equation}

On the other hand, let $\Pi_\rD:\bH^1(\Omega_\rD)\to \bH_h(\Omega_\rD)$ be the well-known Raviart--Thomas interpolation operator \cite{Raviart-Thomas}.
We recall that, given $\bv\in\bH^1(\Omega_\rD)$, this operator is characterized by
\begin{equation}\label{eq:PiD-edge}
\int_e \Pi_\rD(\bv)\cdot\bn \,=\, \int_e \bv\cdot\bn \quad \mbox{ for each edge $e$ of } \cT^\rD_h,
\end{equation}
which implies that
\begin{equation}\label{eq:q-div-PiD}
(\div(\Pi_\rD(\bv)),q_h)_\rD \,=\, (\div(\bv),q_h)_\rD \quad \forall\,q_h\in\L_h(\Omega).
\end{equation}
Equivalently, if $\cP_\rD$ denotes the $\L^2(\Omega_\rD)$-orthogonal projection onto the restriction 
of $\L_h(\Omega)$ to $\Omega_\rD$, then the relation \eqref{eq:q-div-PiD} can be written as
\begin{equation}\label{eq:PD-div}
\div(\Pi_\rD(\bv)) \,=\, \cP_\rD(\div(\bv)) \quad \forall\,\bv\in\bH^1(\Omega_\rD) \,.
\end{equation}

Let us now observe that the set of discrete normal traces on $\Sigma$ of $\bH_h(\Omega_\rD)$ is given by
\begin{equation}\label{eq:Phih-definition}
\Phi_h(\Sigma) \,:=\, \Big\{ \phi_h : \Sigma\to \R : \quad \phi_h|_e\in \rP_0(e) \quad \forall\,\mbox{ edge } e\in \Sigma_h \Big\} \,.
\end{equation}
In \cite[Theorem A.1]{mms2015} it has been proved that there exists a discrete lifting
\begin{equation}\label{eq:Lifting-Lh}
\bL_h: \Phi_h(\Sigma) \to \bH_{h,\Gamma_\rD}(\Omega_\rD) \,,
\end{equation}
such that, for all $\phi_h\in \Phi_h(\Sigma)$,
\begin{equation}\label{eq:Lh-properties}
\|\bL_h(\phi_h)\|_{\div;\Omega_\rD} \,\leq\, C_{\Sigma}\,\|\phi_h\|_{-1/2,\Sigma} \qan
\bL_h(\phi_h)\cdot\bn \,=\, \phi_h \qon \Sigma\,.
\end{equation}
In addition, in \cite[Lemma 5.2]{gos2011} it has been proved that there exits $\wh{\beta}_\Sigma>0$, independent of $h$, such that the pair of subspaces $(\Phi_h(\Sigma),\Lambda_h(\Sigma))$ satisfies the discrete inf-sup condition:
\begin{equation}\label{eq:Lifting-inf-sup-condition}
\sup_{0\neq \phi_h\in \Phi_h(\Sigma)}\frac{\pil\phi_h, \xi_h\pir_\Sigma}{\|\phi_h\|_{-1/2,\Sigma}}
\,\geq\, \wh{\beta}_\Sigma \,\|\xi_h\|_{1/2,\Sigma} \quad \forall\,\xi_h\in \Lambda_h(\Sigma) \,.
\end{equation}

\subsection{Well-posedness of the discrete problem}

Now we prove the well-posedness of problem \eqref{eq:discrete-mixed-formulation}
by employing analogous arguments to the ones developed in Theorem \ref{thm:well-posedness-result}.
We begin by establishing the continuity and strong monotonicity of the operator $\ba$ 
on the discrete kernel of $\bb$:
\begin{equation*}
\bV_h \,:=\, \Big\{ \bv_h :=(\bv_{\rB,h},\bv_{\rD,h})\in \bH_h :\quad [\bb(\bv_h),(q_h,\xi_h)] = 0\quad \forall\,(q_h,\xi_h)\in \bQ_h \Big\} \,.
\end{equation*}

Observe that, similarly to the continuous case, $\bv_h\in \bV_h$ if and only if
\begin{equation*}
(q_h, \div(\bv_{\rB,h}))_\rB + (q_h,\div(\bv_{\rD,h}))_\rD \,=\, 0 \quad \forall\,q_h\in \L_{h,0}(\Omega)\,,
\end{equation*}
and
\begin{equation*}
\pil \bv_{\rB,h}\cdot\bn - \bv_{\rD,h}\cdot\bn,\xi_h \pir_\Sigma \,=\, 0 \quad \forall\,\xi_h\in \Lambda_h(\Sigma) \,,
\end{equation*}
which, in particular imply that
\begin{equation}\label{eq:Vh-condition}
(q_h, \div(\bv_{\rB,h}))_\rB \,=\, 0 \quad \forall\,q_h\in \L_h(\Omega_\rB) \qan
\div(\bv_{\rD,h}) \,=\, 0 \qin \Omega_\rD\,,
\end{equation}
where $\L_h(\Omega_\rB)$ is the set of functions of $\L_h(\Omega)$ restricted to $\Omega_{\rB}$.
In this way, using \eqref{eq:Vh-condition}, we address the discrete counterparts of 
Lemmas \ref{lem:continuity-a} and \ref{lem:strong-monotonicity-a},
whose proofs, being almost verbatim of the continuous ones, are omitted.
\begin{lem}\label{lem:strong-monotonicity-continuity-a-Hh}
Let $\rp\in [3,4]$. Then, the family of operators 
$\Big\{ \ba(\cdot + \bz_h) : \bV_h\to \bV'_h :\quad \bz_h\in \bH_h \Big\}$ 
is uniformly strongly monotone with the same constant $\gamma_{\tBFD}>0$
from \eqref{eq:strict-monotonicity-a}, that is, there holds
\begin{equation*}
[\ba(\bu_h + \bz_h) - \ba(\bv_h + \bz_h),\bu_h - \bv_h] 
\,\geq\, \gamma_{\tBFD}\,\|\bu_h - \bv_h\|^2_{\bH} \,,
\end{equation*}
for each $\bz_h=(\bz_{\rB,h},\bz_{\rD,h})\in \bH_h$, and for all $\bu_h=(\bu_{\rB,h},\bu_{\rD,h}), \bv_h=(\bv_{\rB,h},\bv_{\rD,h})\in \bV_h$.
In addition, the operator $\ba:\bH_h\to \bH'_h$ is continuous in the sense of \eqref{eq:continuity-a},
with the same constant $L_{\tBFD}$.
\end{lem}

We continue with the discrete inf-sup condition of $\bb$.
To that end, we first recall from \cite{do2017}
the inf-sup conditions
\begin{equation}\label{eq:inf-sup-b-aux-1}
\sup_{\0\neq \bv_h\in \bH_h} 
\frac{[\bb(\bv_h),(q_h,\xi_h)]}{\|\bv_h\|_{\bH}}
\,\geq\, \wt{C}_1\,\|\xi_h\|_{1/2,\Sigma} - \|q_h\|_{0,\Omega} \,,
\end{equation}
and
\begin{equation}\label{eq:inf-sup-b-aux-2}
\sup_{\0\neq \bv_h\in \bH_h} 
\frac{[\bb(\bv_h),(q_h,\xi_h)]}{\|\bv_h\|_{\bH}}
\,\geq\, \wt{C}_2\,\|q_h\|_{0,\Omega} - \wt{C}_3\,h^{1/2}_\rD\,\|\xi_h\|_{1/2,\Sigma} \,,
\end{equation}
for all $(q_h,\xi_h)\in \bQ_h$, where $\wt{C}_1$ and $\wt{C}_2, \wt{C}_3$ are positive
constants described in \cite[Lemmas 7 and 8]{do2017}, respectively, and whose proofs follow from 
the use of the lifting $\bL_h$ defined in \eqref{eq:Lifting-Lh}, 
properties \eqref{eq:Lh-properties} and \eqref{eq:Lifting-inf-sup-condition},
and the Bernardi--Raugel and Raviart--Thomas interpolations properties
\eqref{eq:PiB-edge}--\eqref{eq:PB-div} and \eqref{eq:PiD-edge}--\eqref{eq:PD-div}, respectively.
According to the above, after a suitable combination of the inf-sup conditions \eqref{eq:inf-sup-b-aux-1} and \eqref{eq:inf-sup-b-aux-2}, the following result holds (see \cite[Lemma 9]{do2017}).
\begin{lem}\label{lem:inf-sup-operator-b-Qh}
Assume that
\begin{equation}\label{eq:hD-condition}
h_\rD \,\leq\ \left( \frac{\wt{C}_1\,\wt{C}_2}{2\,\wt{C}_3} \right)^2 \,.
\end{equation}	
Then, there exists $\wh{\beta} > 0$, independent of $h$, such that
\begin{equation}\label{eq:inf-sup-b-Qh}
\sup_{\0\neq \bv_h\in \bH_h} 
\frac{[\bb(\bv_h),(q_h,\xi_h)]}{\|\bv_h\|_{\bH}}
\,\geq\, \wh{\beta}\,\|(q_h,\xi_h)\|_{\bQ} \quad
\forall\,(q_h,\xi_h)\in \bQ_h \,.
\end{equation}
\end{lem}

We are now in position to establish the main result of this section,
namely, existence and uniqueness of solution of problem \eqref{eq:discrete-mixed-formulation}.
\begin{thm}
Let $\rp\in [3,4]$.
Assume that \eqref{eq:hD-condition} holds.
Then, the problem \eqref{eq:discrete-mixed-formulation} has a unique 
solution $(\bu_h,(p_h,\lambda_h))\in \bH_h\times \bQ_h$.
Moreover, there exist constants $\wh{C}_1, \wh{C}_2>0$, independent of $h$ and of the solution, such that
\begin{equation}\label{eq:a-priori-bound-uh}
\|\bu_h\|_{\bH} \,\leq\, \wh{C}_1\,\Big( \|\f_\rB\|_{0,\Omega_\rB} + \|\f_\rD\|_{0,\Omega_\rD} 
+ \|g_\rD\|_{0,\Omega_\rD} + \|g_\rD\|^{\rp-1}_{0,\Omega_\rD} \Big) 
\end{equation}
and
\begin{equation}\label{eq:a-priori-bound-ph-lamh}
\|(p_h,\lambda_h)\|_{\bQ} \,\leq\, \wh{C}_2\,\sum_{j\in \{\rp,2\} } \Big( \|\f_\rB\|_{0,\Omega_\rB} + \|\f_\rD\|_{0,\Omega_\rD} 
+ \|g_\rD\|_{0,\Omega_\rD} + \|g_\rD\|^{\rp-1}_{0,\Omega_\rD} \Big)^{j-1}  \,.
\end{equation}
\end{thm}
\begin{proof}
According to Lemma \ref{lem:strong-monotonicity-continuity-a-Hh} and 
the discrete inf-sup condition for $\bb$ provided by \eqref{eq:inf-sup-b-Qh} 
(cf. Lemma \ref{lem:inf-sup-operator-b-Qh}),
the proof follows from a direct application of Theorem \ref{thm:existence-approximation-nonlinear-problems}, with $\rp_1=\rp\in [3,4]$ and $\rp_2=2$,
to the discrete setting represented by \eqref{eq:discrete-mixed-formulation}.
In particular, the {\it a priori} bounds \eqref{eq:a-priori-bound-uh} and 
\eqref{eq:a-priori-bound-ph-lamh} are consequence of the abstract estimates 
\eqref{eq:a-priori-u-thm-nonlinear-problems} and \eqref{eq:a-priori-sigma-thm-nonlinear-problems}, 
respectively, applied to \eqref{eq:discrete-mixed-formulation}, which makes use of 
the bounds for $\f$ and $\g$ given by \eqref{eq:stability-f} and \eqref{eq:stability-g}, 
respectively, thus completing the proof.
\end{proof}

\bigskip

We end this section by observing that the existence of a stable lifting $\bL_h$ satisfying \eqref{eq:Lh-properties} 
and the inf-sup condition \eqref{eq:Lifting-inf-sup-condition} play an important role in 
the proof of the discrete inf-sup condition \eqref{eq:inf-sup-b-aux-1}. In particular, 
as established in Section \ref{sec:discrete-setting}, the existence of a stable lifting $\bL_h$,
satisfying \eqref{eq:Lh-properties}, has been proved in \cite[Theorem A.1]{mms2015} (see also \cite{gos2011} for a similar result) 
for the 2D case, where the only restriction on the grid is shape regularity. 
Now, concerning the existence of a discrete lifting $\bL_h$ in a three dimensional domain,
we refer to \cite{ags2016} for an extension of \cite[Theorem A.1]{mms2015} 
to the 3D case, where again the only requirement on the mesh is shape regularity 
(see \cite[Theorem 2.1]{ags2016}).
However, in order to be able to prove the 3D version of the inf-sup condition \eqref{eq:Lifting-inf-sup-condition},
unlike the 2D case, the discrete subspace $\Lambda_h$ must be defined on an independent
triangulation $\Sigma_{\wt{h}}$ of the interface $\Sigma$ formed by triangles of diameter $\wt{h}_{T}$. 
Then, setting
$\wt{h}_\Sigma := \max\big\{ \wt{h}_T :\, T\in \Sigma_{\wt{h}} \big\}$, and defining the set 
of normal traces of $\bH_h(\Omega_\rD)$ as in \eqref{eq:Phih-definition} 
(considering triangles instead of edges), with $h_\Sigma := \max\big\{ h_T :\, T\in \Sigma_{h} \big\}$, 
it can be proved, by extending previous results on mixed methods with Lagrange multipliers
originally provided in \cite{bg2003}, that there exists $C_0\in (0,1)$ such that 
for each pair $(h_\Sigma, \wt{h}_\Sigma)$ verifying $h_\Sigma \leq C_0\,\wt{h}_\Sigma$, 
the 3D version of \eqref{eq:Lifting-inf-sup-condition} is satisfied, see e.g., the second part of the proof of \cite[Lemma 7.5]{ghm2010}.

\subsection{\textit{A priori} error analysis}

In this section we derive the C\'ea estimate for 
the Galerkin scheme \eqref{eq:discrete-mixed-formulation} 
with the finite element subspaces given by \eqref{eq:FEM-1}--\eqref{eq:FEM-2}, and then use 
the approximation properties of the latter to establish the corresponding rates of convergence.
In fact, let $(\bu,(p,\lambda))\in \bH\times \bQ$ and $(\bu_h,(p_h,\lambda_h))\in \bH_h\times \bQ_h$,  
be the unique solutions of the continuous and discrete coupled problems
\eqref{eq:mixed-variational-formulation} and \eqref{eq:discrete-mixed-formulation}, respectively.
Then, we are interested in obtaining an {\it a priori} estimate for the global error
\begin{equation*}
\|\bu - \bu_h\|_{\bH} \,+\, \|(p,\lambda) - (p_h,\lambda_h)\|_{\bQ} \,.
\end{equation*}
For this purpose, we establish next a slight adaptation of the Strang-type estimate 
provided in \cite[Lemma 5.1]{cgo2021}.
Hereafter, given a subspace $X_h$ of a generic Banach space 
$(X,\|\cdot\|_X)$, we set as usual
\begin{center}
$\ds\dist(x,X_h) := \inf_{x_h\in X_h} \|x - x_h\|_{X}$ for all $x\in X$.
\end{center}
\begin{lem}\label{lem:Strang-prob-nonlinear}
Let $X_1, X_2$ and $Y$ be separable and reflexive Banach spaces,
being $X_1$ and $X_2$ uniformly convex, and set $X:=X_1\times X_2$.
Let $\mathcal{A}:X\to X'$ 
be a nonlinear operator and $\mathcal{B}\in \mathcal{L}(X, Y')$, 
such that $\mathcal{A}$ and $\mathcal{B}$ satisfy the hypotheses of 
Theorem {\rm \ref{thm:existence-approximation-nonlinear-problems}} with 
respective constants $L$, $\alpha$, $\beta$, and exponents $p_{1}, p_2 \ge 2$. 
Furthermore, let $\{X_{1,h}\}_{h>0}, \{X_{2,h}\}_{h>0}$ and $\{Y_{h}\}_{h>0}$ 
be sequences of finite dimensional subspaces of $X_1, X_2$, and $Y$, 
respectively. Set $X_h:=X_{1,h}\times X_{2,h}$, and 
consider $\mathcal{A}|_{X_h}:X_h \to X'_h$ 
and $\mathcal{B}|_{X_{h}}: X_{h}\to Y_{h}'$ 
satisfying the hypotheses of Theorem {\rm \ref{thm:existence-approximation-nonlinear-problems}} 
as well, with constants $L_{\ttd}, \alpha_{\ttd}$, and $\beta_{\ttd}$, 
all of them independent of $h$. Finally, given $\mathcal{F}\in X'$, 
$\mathcal{G}\in Y'$, 
we let $(u,\sigma) = ((u_1,u_2),\sigma)\in X\times Y$ and 
$(u_{h},\sigma_{h}) = ((u_{1,h},u_{2,h}),\sigma_{h})\in X_{h}\times Y_{h}$ 
be the unique solutions, respectively, to the problems
\begin{equation}\label{eq:teo-Strang-nolinear}
\begin{array}{rcll}
[\cA(u),v] + [\cB(v),\sigma] & = & [\mathcal{F},v] & \forall\,v\in X\,, \\ [2ex]
[\cB(u),\tau] & = & [\mathcal{G},\tau] & \forall\,\tau\in Y\,,
\end{array}
\end{equation}
and 
\begin{equation}\label{eq:teo-Strang-discrete-nolinear}
\begin{array}{rcll}
[\cA(u_h),v_h] + [\cB(v_h),\sigma_h] & = & [\mathcal{F},v_h] & \forall\,v_h\in X_h\,, \\ [2ex]
[\cB(u_h),\tau_h] & = & [\mathcal{G},\tau_h] & \forall\,\tau_h\in Y_h\,.
\end{array}
\end{equation}
Then, there exists a positive constant $C_{ST}$, depending only on $p_1$, $p_2$,
$L_{\ttd}$, $\alpha_{\ttd}$, $\beta_{\ttd}$, and $\|\cB\|$, such that
\begin{equation*}
\begin{array}{l}
\ds \|u - u_{h}\|_X + \|\sigma - \sigma_h\|_Y \\[1ex]
\ds\quad \leq\, C_{ST}\,C_1(u,u_h) \,
\Big\{ C_2(u) \,\dist(u,X_{h}) \,+\, \sum_{j=1}^{2}\dist(u,X_{h})^{p_{j}-1}
\,+\, \dist(\sigma,Y_{h}) \Big\}\,,
\end{array}
\end{equation*}
where
\begin{equation*}
C_1(u, u_h) := 1 + \sum_{j=1}^{2}\big(\|u_{j}\|_{X_j} + \|u_{j,h}\|_{X_j} \big)^{p_j - 2}
\qan
C_2(u) := 1 + \sum_{j=1}^{2}\|u_{j}\|_{X_j}^{p_{j}-2} \,.
\end{equation*}
\end{lem}

We now establish the main result of this section.
\begin{thm}\label{thm:cea-estimate}
Let $\rp\in [3,4]$.
Assume that \eqref{eq:hD-condition} holds. 
Let $(\bu,(p,\lambda)) := ((\bu_\rB,\bu_\rD),(p,\lambda)) \in \bH\times\bQ$ and 
$(\bu_h,(p_h,\lambda_h)) := ((\bu_{\rB,h},\bu_{\rD,h}),(p_h,\lambda_h)) \in \bH_h\times \bQ_h$ 
be the unique solutions of the continuous and discrete problems \eqref{eq:mixed-variational-formulation} and \eqref{eq:discrete-mixed-formulation}, respectively. 
Then, there exists $C>0$, independent of $h$ and of the continuous and discrete solutions, such that
\begin{equation}\label{eq:cea-estimate}
\|\bu - \bu_h\|_{\bH} \,+\, \|(p,\lambda) - (p_h,\lambda_h)\|_{\bQ}
\,\leq\, C\,\Big\{ \sum_{j\in \{\rp,2\}} \dist(\bu,\bH_h)^{j-1} + \dist((p,\lambda),\bQ_h) \Big\}.
\end{equation}
\end{thm}
\begin{proof}
First, note that the continuous and discrete problems \eqref{eq:mixed-variational-formulation} 
and \eqref{eq:discrete-mixed-formulation} have the structure of 
\eqref{eq:teo-Strang-nolinear} and \eqref{eq:teo-Strang-discrete-nolinear}, respectively.
Thus, as a direct application of Lemma \ref{lem:Strang-prob-nonlinear}, we obtain
\begin{equation*}
\begin{array}{l}
\ds \|\bu - \bu_h\|_{\bH} \,+\, \|(p,\lambda) - (p_h,\lambda_h)\|_{\bQ} \\[2ex]
\ds \quad \leq\, C_{ST}\,C_1(\bu,\bu_h) \,
\Big\{ C_2(\bu) \,\dist(\bu,\bH_{h}) \,+\, \sum_{j\in \{\rp,2\}} \dist(\bu,\bH_{h})^{j-1}
+ \, \dist((p,\lambda),\bQ_{h}) \Big\}\,,
\end{array}
\end{equation*}
where
\begin{equation*}
C_1(\bu,\bu_h) := 1 + \|\bu_\rB\|^{\rp-2}_{1,\Omega_\rB} + \|\bu_{\rB,h}\|^{\rp-2}_{1,\Omega_\rB}
\qan 
C_2(\bu) := 1 + \|\bu_\rB\|^{\rp-2}_{1,\Omega_\rB}
\end{equation*}
are bounded by data thanks to the {\it a priori} bounds \eqref{eq:a-priori-bound-u} and \eqref{eq:a-priori-bound-uh}. This yields \eqref{eq:cea-estimate} and concludes the proof.
\end{proof}

Now, in order to provide the theoretical rate of convergence of the Galerkin scheme
\eqref{eq:discrete-mixed-formulation}, we recall the approximation properties of the
subspaces involved (see, e.g., \cite{br1985,Brezzi-Fortin,Ern-Guermond,Gatica}).
Note that each one of them is named after the unknown to which it is applied later on.

\bigskip

\noindent $(\mathbf{AP}^{\bu_\rB}_h)$ For each $\bv_\rB\in\bH^2(\Omega_\rB)$, there holds
\begin{equation*}
\|\bv_\rB - \Pi_\rB(\bv_\rB)\|_{1,\Omega_\rB} 
\,\leq\, C\,h\,\|\bv_\rB\|_{2,\Omega_\rB} \,.
\end{equation*}

\noindent $(\mathbf{AP}^{\bu_\rD}_h)$ For each $\bv_\rD\in\bH^1(\Omega_\rD)$ with $\div(\bv_\rD) \in\H^1(\Omega_\rD)$, there holds
\begin{equation*}
\|\bv_\rD - \Pi_\rD(\bv_\rD)\|_{\div;\Omega_\rD} 
\,\leq\, C\,h\,\Big\{ \|\bv_\rD\|_{1,\Omega_\rD} + \|\div(\bv_\rD)\|_{1,\Omega_\rD} \Big\}\,.
\end{equation*}

\noindent $(\mathbf{AP}^{p}_h)$ For each $q \in\H^1(\Omega)\cap \L^2_0(\Omega)$, there exists $q_h\in\L_{h,0}(\Omega)$ such that
\begin{equation*}
\|q - q_h\|_{0,\Omega} 
\,\leq\, C\,h\,\|q\|_{1,\Omega}\,.
\end{equation*}

\noindent $(\mathbf{AP}^{\lambda}_h)$ For each $\xi\in \H^{3/2}(\Sigma)$, there exists $\xi_h\in \Lambda_h(\Sigma)$ such that
\begin{equation*}
\|\xi - \xi_h\|_{1/2,\Sigma} 
\,\leq\, C\,h\,\|\xi\|_{3/2,\Sigma} \,.
\end{equation*}

The following theorem provides the theoretical rate of convergence of the Galerkin
scheme \eqref{eq:discrete-mixed-formulation}, under suitable regularity assumptions on the exact solution.
Notice that, optimal rates of convergences are obtained for all the unknowns.
\begin{thm}\label{thm:rate-of-convergence}
Let $\rp\in [3,4]$.
Assume that \eqref{eq:hD-condition} holds. 
Let $(\bu,(p,\lambda)) := ((\bu_\rB,\bu_\rD),(p,\lambda)) \in \bH\times\bQ$ and 
$(\bu_h,(p_h,\lambda_h)) := ((\bu_{\rB,h},\bu_{\rD,h}),(p_h,\lambda_h)) \in \bH_h\times \bQ_h$ 
be the unique solutions of the continuous and discrete problems \eqref{eq:mixed-variational-formulation} 
and \eqref{eq:discrete-mixed-formulation}, respectively, and assume that 
$\bu_\rB\in\bH^2(\Omega_\rB)$, $\bu_\rD\in\bH^1(\Omega_\rD)$, $\div(\bu_\rD)\in\H^1(\Omega_\rD)$, 
$p\in\H^1(\Omega)$, and $\lambda\in\H^{3/2}(\Sigma)$. 
Then, there exists $C>0$, independent of $h$ and the continuous and discrete solutions, such that
\begin{equation*}
\begin{array}{l}
\ds \|\bu - \bu_h\|_{\bH} \,+\, \|(p,\lambda) - (p_h,\lambda_h)\|_{\bQ} \\ [2ex]
\ds\quad \leq\, C\,(h + h^{\rp-1})\,\Bigg\{  \sum_{j\in \{\rp,2\}} \big(\|\bu_\rB\|_{2,\Omega_\rB} + \|\bu_\rD\|_{1,\Omega_\rD} + \|\div(\bu_\rD)\|_{1,\Omega_\rD} \big)^{j-1}
+ \|p\|_{1,\Omega} + \|\lambda\|_{3/2,\Sigma} \Bigg\} \,.
\end{array}
\end{equation*}
\end{thm}
\begin{proof}
The result follows from a direct application of Theorem \ref{thm:cea-estimate} and the approximation properties of the discrete subspaces. Further details are omitted.
\end{proof}


\section{Numerical results}\label{sec:numerical-results}

In this section we present two examples illustrating the performance of 
the mixed finite element scheme \eqref{eq:discrete-mixed-formulation} on a set of 
quasi-uniform triangulations of the corresponding domains. 
Our implementation is based on a {\tt FreeFem++} code \cite{freefem}, 
in conjunction with the direct linear solver {\tt UMFPACK} \cite{umfpack}. 
In order to solve the nonlinear problem \eqref{eq:discrete-mixed-formulation}, 
given $\0\neq \bw_\rB\in\bH^1_{\Gamma_\rB}(\Omega_\rB)$ we introduce the G\^ateaux derivative associated to $\ba$ (cf. \eqref{eq:definition-a}), i.e.,
\begin{align*}
[\cD\ba(\bw_\rB)(\bu),\bv] 
\,:=\,\, & \mu\,(\nabla\bu_\rB,\nabla\bv_\rB)_\rB 
+ (\bK^{-1}_\rB\bu_\rB,\bv_\rB)_\rB
+ \tF\,( |\bw_\rB|^{\rp-2}\bu_\rB,\bv_\rB )_\rB \\[2ex]
& +\, \tF\,(\rp-2)( |\bw_\rB|^{\rp-4}(\bw_\rB\cdot\bu_\rB)\bw_\rB,\bv_\rB )_\rB
+ \left(\bK^{-1}_\rD \bu_\rD,\bv_\rD\right)_\rD \,,
\end{align*}
for all $\bu,\bv \in\bH$.
In this way, we propose the Newton-type strategy: Given $\0\neq\bu^0_{\rB,h}\in \bH_{h,\Gamma_\rB}(\Omega_\rB)$, for $m\geq 1$, find 
$\bu^m_h = (\bu^m_{\rB,h},\bu^m_{\rD,h})\in\bH_h$ and $(p^m_h, \lambda^m_h)\in \bQ_h$, such that
\begin{equation}\label{eq:iterative-newton-method}
\begin{array}{lll}
[\cD\ba(\bu^{m-1}_{\rB,h})(\bu^m_h),\bv_h] + [\bb(\bv_h),(p^m_h,\lambda^m_h)] & = & [\f,\bv_h] + \tF\,(\rp-2)\,( |\bu^{m-1}_{\rB,h}|^{\rp-2}\bu^{m-1}_{\rB,h},\bv_{\rB,h} )_\rB \,, \\ [2ex]
[\bb(\bu^m_h),(q_h,\xi_h)] & = & [\g,(q_h,\xi_h)] \,,
\end{array}
\end{equation}
for all $\bv_h = (\bv_{\rB,h},\bv_{\rD,h})\in\bH_h$ and $(q_h,\xi_h)\in\bQ_h$.

The iterative method is stopped when the relative error between 
two consecutive iterations of the complete coefficient vector, namely $\coeff^m$ 
and $\coeff^{m+1}$, is sufficiently small, that is
\begin{equation*}
\frac{\|\coeff^{m+1} - \coeff^m\|_{\ell^2}}{\|\coeff^{m+1}\|_{\ell^2}} \,\leq\, \tol \,,
\end{equation*}
where $\|\cdot\|_{\ell^2}$ is the standard $\ell^2$-norm in $\R^{\DOF}$, 
with $\DOF$ denoting the total number of degrees of
freedom defining the finite element subspaces $\bH_{h,\Gamma_\rB}(\Omega_{\rB}), \bH_{h,\Gamma_\rD}(\Omega_{\rD}), \L_{h,0}(\Omega)$, and $\Lambda_h(\Sigma)$, 
and $\tol$ is a fixed tolerance chosen as $\tol=1\textup{E}-06$.

The errors for each variable are denoted by:
\begin{equation*}
\begin{array}{c}
\re(\bu_\rB) \,:=\, \|\bu_\rB - \bu_{\rB,h}\|_{1,\Omega_\rB}\,,\quad 
\re(\bu_\rD) \,:=\, \|\bu_\rD - \bu_{\rD,h}\|_{\div;\Omega_\rD} \,, \\ [2ex]
\re(p_\rB) \,:=\, \|p_\rB - p_{\rB,h}\|_{0,\Omega_\rB}\,,\quad
\re(p_\rD) \,:=\, \|p_\rD - p_{\rD,h}\|_{0,\Omega_\rD}\,,\quad 
\re(\lambda) \,:=\, \|\lambda - \lambda_h\|_{1/2,\Sigma} \,.
\end{array}
\end{equation*}
Notice that, for ease of computation, the interface norm 
$\|\lambda - \lambda_h\|_{1/2,\Sigma}$ will be replaced by $\|\lambda - \lambda_h\|_{(0,1),\Sigma}$ with
\begin{equation*}
\|\xi\|_{(0,1),\Sigma} \,:=\, \|\xi\|^{1/2}_{0,\Sigma} \, \|\xi\|^{1/2}_{1,\Sigma} 
\quad \forall\,\xi\in \H^1(\Sigma) \,,
\end{equation*}
owing to the fact that $\H^{1/2}(\Sigma)$ is the interpolation space with index $1/2$
between $\H^1(\Sigma)$ and $\L^2(\Sigma)$.

Moreover, the respective experimental rates of convergence are computed as
\begin{equation*}
\sr(\diamond) \,:=\, \frac{\log(\re(\diamond)/\wh{\re}(\diamond))}{\log(h/\wh{h})} \quad 
\mbox{for each } \diamond\in\Big\{ \bu_\rB, \bu_\rD, p_\rB, p_\rD, \lambda \Big\}\,,
\end{equation*}
where $h$ and $\wh{h}$ denote two consecutive mesh sizes with errors $\re$ and $\wh{\re}$, respectively.

For each example shown below we take $\bu^0_{\rB,h} =(0.1,0)$ as initial guess.
In addition, the condition $(p_h,1)_{\Omega} = 0$ is imposed via a penalization strategy.

\subsection*{Example 1: Tombstone-shaped domain with varying $\tF, \bK_{\rB}$, and $\bK_{\rD}$ parameters.}

In our first example, we validate the rates of convergence in a two-dimensional domain
and also study the performance of the numerical method with respect to the number of
Newton iterations when different
values of the parameters $\tF$, $\bK_{\rB}$ and $\bK_{\rD}$ are considered.
More precisely, we consider a semi-disk-shaped porous domain coupled with a porous unit square, 
i.e.,
\begin{equation*}
\Omega_\rB := \Big\{ (x_1,x_2) :\quad x^2_1 + (x_2-0.5)^2 < 0.5^2,\,\, x_2>0.5 \Big\} \qan \Omega_\rD := (-0.5,0.5)^2\,,
\end{equation*}
with interface $\Sigma := (-0.5,0.5)\times\{0.5\}$.
We consider the model parameter $\rp=3$, $\mu=1$, $\tF=10$, $\bK_{\rB}=\bI$, $\bK_{\rD}=10^{-1}\,\bI$, and the data $\f_\rB, \f_\rD$, and $g_\rD$ are chosen so that  the exact solution in the tombstone-shaped porous domain $\Omega = \Omega_\rB\cup\Sigma\cup\Omega_\rD$ 
is given by the smooth functions
\begin{equation*}
\begin{array}{c}
\bu_\rB(x_1,x_2) := \left(\begin{array}{r}
\cos(\pi x_1)\sin(\pi x_2) \\ -\sin(\pi x_1)\cos(\pi x_2)
\end{array}\right),\quad 
\bu_\rD(x_1,x_2) := \left(\begin{array}{r}
\cos(\pi x_1)\exp(x_2) \\ \exp(x_1)\cos(\pi x_2)  
\end{array}\right),\\ [3ex]
p_\star(x_1,x_2) := \sin(\pi x_1)\sin(\pi x_2) \qin \Omega_\star,\quad \mbox{ with } \star\in\{\rB,\rD\}.
\end{array}
\end{equation*}
Notice that this solution satisfies $\bu_\rB\cdot\bn = \bu_\rD\cdot\bn$ on $\Sigma$. 
However, the second transmission condition in \eqref{eq:transmission-condition} is not satisfied, 
and the Dirichlet boundary condition for the Brinkman--Forchheimer velocity on $\Gamma_\rB$ and 
the Neumann boundary condition for the Darcy velocity on $\Gamma_\rD$ are both 
non-homogeneous. This gives rise to additional contributions that are included in the right-hand side of the resulting system. 
The results reported in Table \ref{table1-ex1-rate} agree with 
the theoretical optimal rate of convergence $O(h)$ provided by Theorem \ref{thm:rate-of-convergence}.
Some components of the numerical solution are displayed in Figure \ref{fig:Example-1}, 
and they were computed using the mixed $\bBR - \bRT_0 - \rP_0$
approximation with mesh size $h=0.013$ and $53,511$ triangle elements (corresponding to 
$148,928\,\DOF$). 
We observe that the continuity of the normal trace of the velocities on $\Sigma$ is preserved since the second components of $\bu_\rB$ and $\bu_\rD$ do coincide on $\Sigma$ as expected. 
It can also be seen that the pressure is continuous in the whole domain and preserves its sinusoidal behavior.

In Table \ref{table1-ex1-Newton}, we report the number of Newton iterations as a function of the parameters $\tF$, $\bK_{\rB}$ and $\bK_{\rD}$, considering different mesh sizes $h$. We can observe that Newton's method is robust with respect to both $h$ and $\bK_{\rD}$, while the number of iterations increases for larger values of $\tF$ due to the increased weight of the nonlinear term $\tF\,|\bu_\rB|\bu_\rB$ in the Brinkman--Forchheimer model. 

\subsection*{Example 2: Flow through a heterogeneous porous media.}

In our second example, we study the behavior of the numerical method for different values of $\tF$ when $\rp = 4$ to model the higher-order inertial correction $\tF\,|\bu_{\rB}|^{2}\bu_{\rB}$ discussed, e.g., in \cite{Firdaouss:1997:jfm}.
We consider the rectangular domain $\Omega=\Omega_{\rB}\cup \Sigma\cup \Omega_{\rD}$,
where 
\begin{equation*}
\Omega_\rB :=(0,2)\times (0,1),\quad \Sigma:=(0,2)\times \{0\},\qan \Omega_\rD:=(0,2)\times (-1,0)\,,
\end{equation*}
with boundaries $\Gamma_\rB = \Gamma_{\rB,\textrm{left}}\cup \Gamma_{\rB,\textrm{top}}\cup \Gamma_{\rB,\textrm{right}}$ and $\Gamma_\rD=\Gamma_{\rD,\textrm{left}}\cup \Gamma_{\rD,\textrm{bottom}}\cup \Gamma_{\rD,\textrm{right}}$, respectively. 
The problem parameters are $\mu=1$, $\bK_{\rB}=10^{-1}\,\bbI$ and $\bK_{\rD}=10^{-3}\,\bI$.
The right-hand side data $\f_\rB, \f_{\rD}$, and $g_\rD$ are chosen as zero, and the boundary conditions are
\begin{equation*}
\begin{array}{l}
\ds \bu_\rB = (-10\,x_2\,(x_2 - 1), 0)^\rt  \qon \Gamma_{\rB,\textrm{left}}\,,\quad
\bu_{\rB} = \0 \qon \Gamma_{\rB,\textrm{top}} \,,\quad
\bsi_\rB\bn = \0 \qon \Gamma_{\rB,\textrm{right}} \,, \\[2ex]
\ds p_\rD = 0 \qon \Gamma_{\rD,\textrm{bottom}} \,,\quad
\bu_\rD\cdot\bn = 0 \qon \Gamma_{\rD,\textrm{left}}\cup \Gamma_{\rD,\textrm{right}} \,.
\end{array}
\end{equation*}

In Figure~\ref{fig:Example-2}, we plot the magnitude of the second component of the velocity in the whole domain for $\tF\in \big\{ 0, 10^{0}, 10^{1}, 10^{2}, 10^{3}, 10^{4} \big\}$, computed using the mixed approximation \eqref{eq:iterative-newton-method} on a mesh with $37,238$ triangular elements (corresponding to $112,771\,\DOF$). As expected, we observe that most of the flow is moving from left to right in the more permeable Brinkman--Forchheimer domain while part of it is driven into the less permeable Darcy medium due to zero pressure at the bottom of the domain. 
For all considered values of $\tF$, the continuity of the normal velocity across the interface is preserved illustrating the mass conservation on $\Sigma$. 
Finally, we notice that, when $\tF$ increases, the magnitude of the vertical component of the velocity decreases at the interface. The number of Newton iterations for the different values of $\tF$ is $\{1, 4, 5, 6, 7, 8\}$, respectively, and we notice that it increases when $\tF$ becomes larger in agreement with what observed in Example 1. (When $\tF=0$ the problem becomes linear, hence only one Newton iteration is performed). 

\bigskip\medskip

\begin{table}[ht]
\begin{center}		
\begin{tabular}{r|c|c||c|c|c|c}
\hline
$\DOF$  & $h_\rB$ & $\iter$ & $\re(\bu_\rB)$ & $\sr(\bu_\rB)$ & $\re(p_\rB)$ & $\sr(p_\rB)$ \\  \hline \hline
172    & 0.330 & 4 & 0.269 &   --  & 1.351 &  --   \\ 
646    & 0.192 & 4 & 0.140 & 1.198 & 0.086 & 5.063 \\ 
2398   & 0.091 & 4 & 0.072 & 0.892 & 0.027 & 1.533 \\ 
9373   & 0.049 & 4 & 0.037 & 1.042 & 0.013 & 1.157 \\ 
37434  & 0.024 & 4 & 0.017 & 1.157 & 0.006 & 1.116 \\ 
148928 & 0.013 & 4 & 0.009 & 1.037 & 0.003 & 1.087 \\ 
\hline 
\end{tabular}
			
\medskip
			
\begin{tabular}{c||c|c|c|c||c||c|c}
\hline
$h_\rD$ & $\re(\bu_\rD)$ & $\sr(\bu_\rD)$ & $\re(p_\rD)$ & $\sr(p_\rD)$ & $h_\Sigma$ & $\re(\lambda)$ & $\sr(\lambda)$ \\  \hline \hline
0.373 & 0.729 &   --  & 2.112 &   --  & 1/2  & 2.337 &   --  \\ 
0.190 & 0.322 & 1.217 & 0.106 & 4.451 & 1/4  & 0.246 & 3.249 \\ 
0.098 & 0.165 & 1.004 & 0.033 & 1.760 & 1/8  & 0.072 & 1.765 \\ 
0.054 & 0.084 & 1.114 & 0.015 & 1.258 & 1/16 & 0.025 & 1.554 \\ 
0.025 & 0.042 & 0.912 & 0.008 & 0.934 & 1/32 & 0.008 & 1.547 \\ 
0.014 & 0.021 & 1.282 & 0.004 & 1.288 & 1/64 & 0.003 & 1.537 \\ 
\hline 
\end{tabular}
\caption{[{\sc Example 1}] Degrees of freedom, mesh sizes, Newton iteration count, errors, and convergence history for the approximation of the coupled Brinkman--Forchheimer/Darcy problem with $\tF=10$, $\bK_{\rB}=\bI$, $\bK_{\rD}=10^{-1}\bI$, and $\rp=3$.}
\label{table1-ex1-rate}
\end{center}
\end{table}

\begin{table}[ht]
\begin{center}
\begin{tabular}{c|c|c||c|c|c|c|c|c}
\hline
$\tF$ & $\bK_{\rB}$ & $\bK_{\rD}$ & $h=0.373$ & $h=0.192$ & $h=0.098$ & $h=0.054$ & $h=0.025$ & $h=0.014$ \\  \hline \hline
$10$  & $1$ & $10^{-1}$ & 4 & 4 & 4 & 4 & 4 & 4 \\ 
$10$  & $1$ & $10^{-2}$ & 4 & 4 & 4 & 4 & 4 & 4 \\ 
$10$  & $1$ & $10^{-3}$ & 4 & 4 & 4 & 4 & 4 & 4 \\ 
$10$  & $1$ & $10^{-4}$ & 3 & 4 & 4 & 4 & 4 & 4 \\ 
$1$    & $1$ & $10^{-1}$ & 4 & 4 & 4 & 4 & 4 & 4 \\ 
$10^2$ & $1$ & $10^{-1}$ & 6 & 6 & 6 & 6 & 6 & 6 \\ 
$10^3$ & $1$ & $10^{-1}$ & 8 & 8 & 8 & 8 & 8 & 8 \\ 
$10^4$ & $1$ & $10^{-1}$ & 9 & 9 & 9 & 9 & 9 & 9 \\ 
\hline 
\end{tabular}
\caption{[{\sc Example 1}] Number of Newton iterations for different values of $\tF$, $\bK_{\rB}$ and $\bK_{\rD}$.}\label{table1-ex1-Newton}
\end{center}
\end{table}

\begin{figure}[ht!]
\begin{center}
\includegraphics[width=4.5cm]{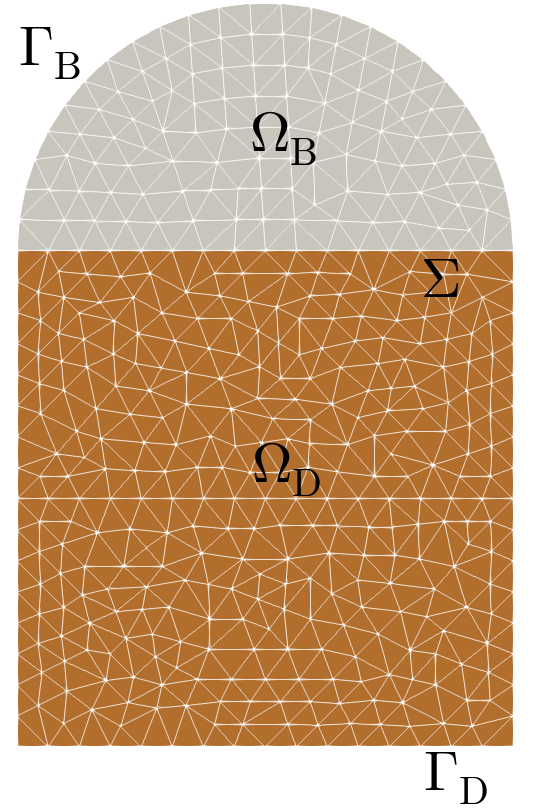}
\hspace{1cm}
\includegraphics[width=4.5cm]{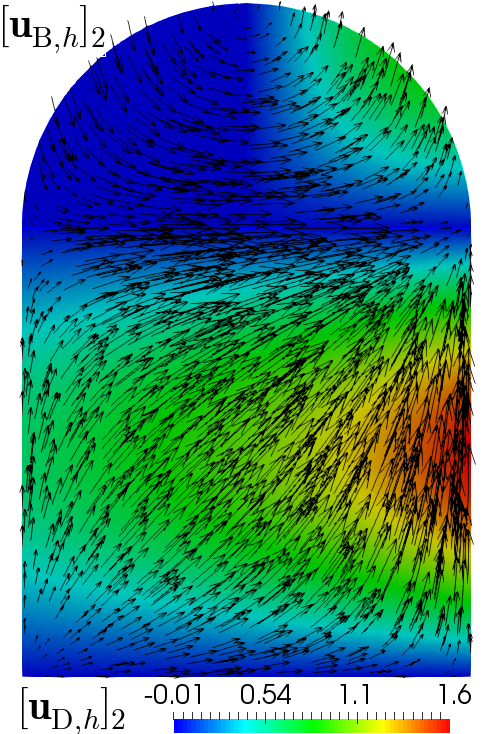}
\hspace{1cm}
\includegraphics[width=4.5cm]{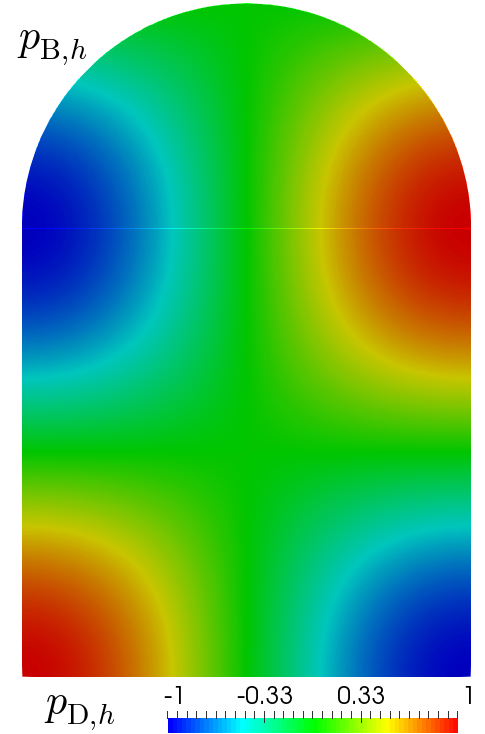}
		
\caption{[{\sc Example 1}] Domain configuration, computed velocity field and magnitude of its second component, and pressure field in the whole domain.}\label{fig:Example-1}
\end{center}
\end{figure}

\begin{figure}[ht!]
\begin{center}
\includegraphics[width=5.5cm]{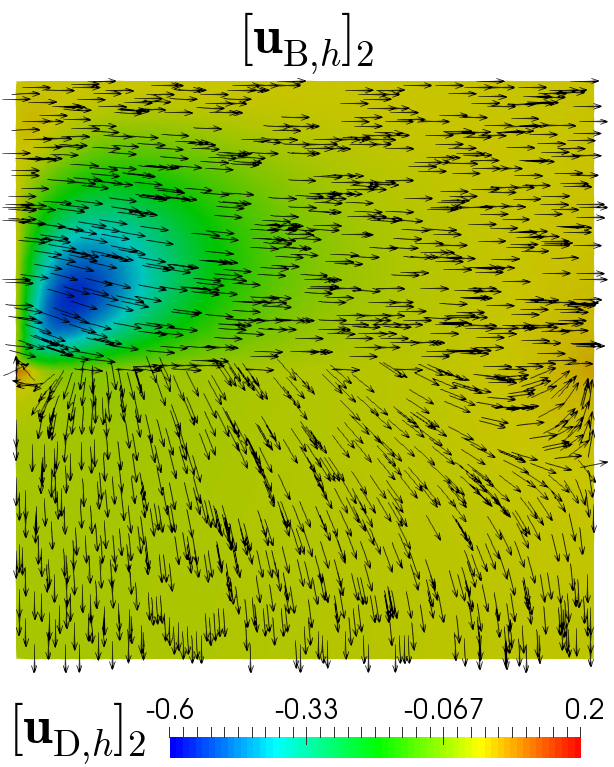}
\includegraphics[width=5.5cm]{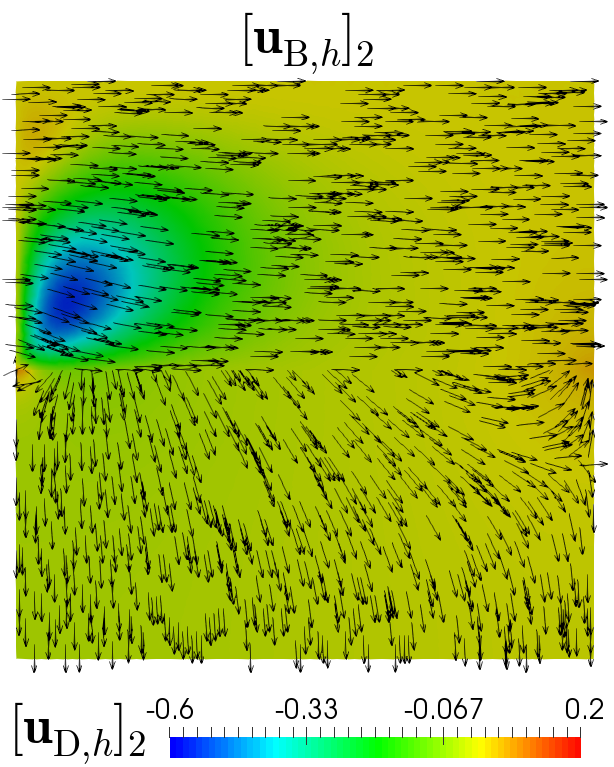}
\includegraphics[width=5.5cm]{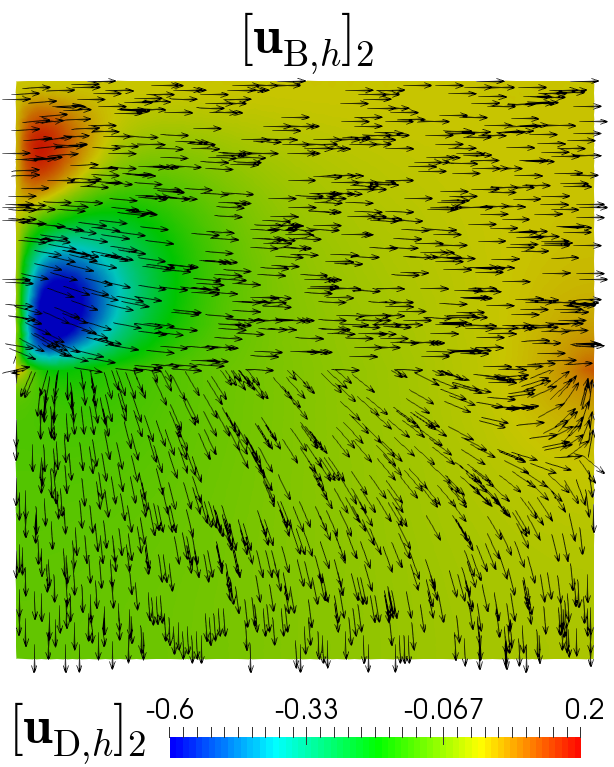}

\vspace{1cm}

\includegraphics[width=5.5cm]{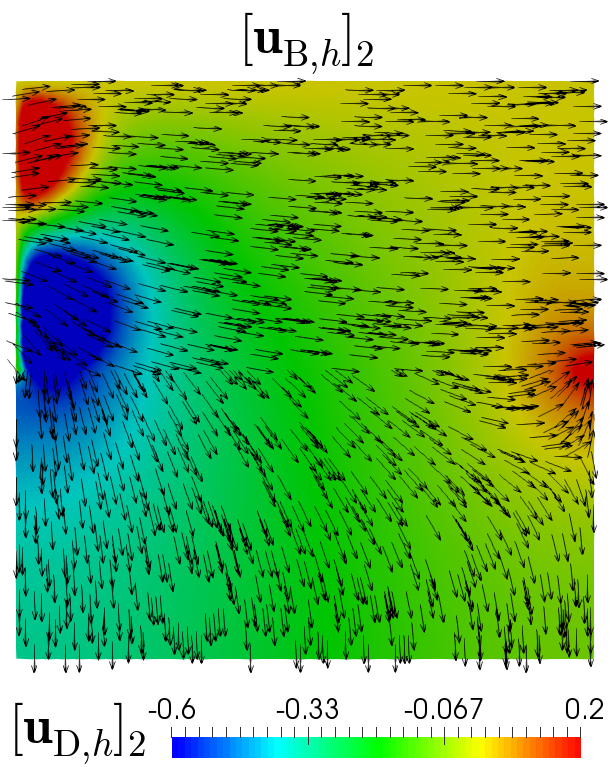}
\includegraphics[width=5.5cm]{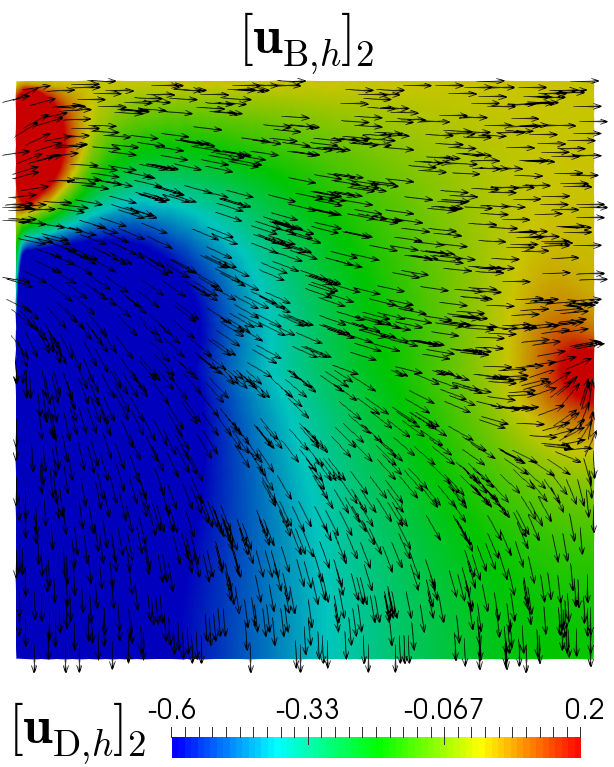}
\includegraphics[width=5.5cm]{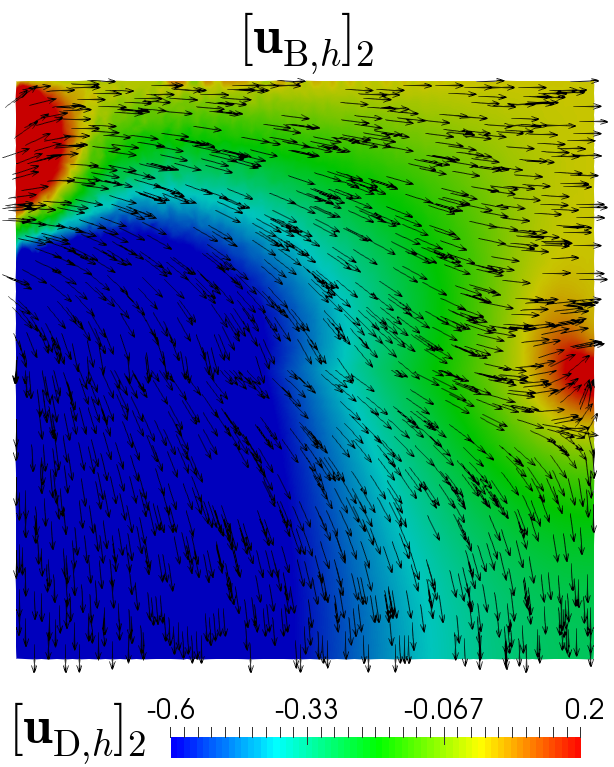}
		
\caption{[{\sc Example 2}] From left to right and from up to-down: magnitude of the second component of the velocity in the whole domain for $\tF\in \{0,10^0,10^1,10^2,10^3,10^4\}$.}\label{fig:Example-2}
\end{center}
\end{figure}


\end{document}